%% file: morse_fun.tex
\documentclass[a4paper,10pt]{amsart}

\usepackage[utf8]{inputenc}
\usepackage[T1]{fontenc}
\usepackage{amsfonts}
\usepackage{amsthm}
\usepackage{amsmath}
\usepackage[english]{babel}
\usepackage[all]{xy}
\usepackage{tikz-cd}
\usepackage{wrapfig}

\usepackage{ amssymb }
\usepackage{graphicx}
\usepackage{amscd}
\usepackage{latexsym}
\usepackage{hyperref}
\usepackage{mathrsfs}
\usepackage{enumerate}

\graphicspath{{pictures/}}
\usepackage{xcolor}

\usepackage{calc}

\newtheorem{theoremABC}{Theorem}

\newtheorem{corABC}{Corollary}

\theoremstyle{definition}
 \newtheorem{defi}{Definition}[section]

\theoremstyle{remark}
 \newtheorem{remark}[defi]{Remark}

\theoremstyle{plain}
\newtheorem{theo}[defi]{Theorem}
\newtheorem{prop}[defi]{Proposition}
\newtheorem{propdef}[defi]{Proposition-Definition}

\newtheorem{lemma}[defi]{Lemma}

\newtheorem*{acknow}{Acknowledgments}

\def\e#1\e{\begin{equation}#1\end{equation}}
\def\ea#1\ea{\begin{align}#1\end{align}}

\numberwithin{equation}{section}

\newcommand{\dd}{\mathbb{D}}

\newcommand{\rr}{\mathbb{R}}

\newcommand{\zz}{\mathbb{Z}}

\newcommand{\ddhat}{\widehat{\mathbb{D}}}

\newcommand{\Z}[1]{\zz_{#1}}

\newcommand{\abs}[1]{\left\vert #1 \right\vert}

\newcommand{\red}{/\!\!/}

\newcommand{\Acal}{\mathcal{A}}
\newcommand{\Bcal}{\mathcal{B}}
\newcommand{\Ccal}{\mathcal{C}}
\newcommand{\Dcal}{\mathcal{D}}

\newcommand{\Gcal}{\mathcal{G}}
\newcommand{\Hcal}{\mathcal{H}}

\newcommand{\Mcal}{\mathcal{M}}

\newcommand{\Pcal}{\mathcal{P}}

\newcommand{\Scal}{\mathcal{S}}

\newcommand{\Ucal}{\mathcal{U}}

\newcommand{\Xfrak}{\mathfrak{X}}

\newcommand{\Xreg}{\mathfrak{X}^{\rm reg}}

\newcommand{\Lund}{\underline{L}}

\newcommand{\kund}{\underline{k}}
\newcommand{\lund}{\underline{l}}

\newcommand{\tund}{\underline{t}}

\newcommand{\kundb}{\underline{k}^{\flat}}
\newcommand{\kundl}{\underline{k}^{l}}
\newcommand{\kundr}{\underline{k}^{r}}

\newcommand{\gamund}{\underline{\gamma}}

\newcommand{\xund}{\underline{x}}

\newcommand{\One}{\underline{1}}

\newcommand{\Crit}{\mathrm{Crit}}

\newcommand{\ind}{\mathrm{ind}}

\renewcommand{\d}{\mathrm{d}}

\renewcommand{\Vert}{\mathrm{Vert}}
\newcommand{\Edges}{\mathrm{Edges}}

\newcommand{\Leaves}{\mathrm{Leaves}}
\newcommand{\Roots}{\mathrm{Roots}}

\newcommand{\Ainf}{$A_\infty$}

\newcommand{\Fuk}{\mathcal{F}uk}

\newcommand{\Lie}{\mathcal{L}ie}

\newcommand{\LieMon}{\mathcal{L}ie\mathcal{M}on}
\newcommand{\LieMonpert}{\mathcal{L}ie\mathcal{M}on^{\rm pert}}

\newcommand{\Lier}{\mathcal{L}ie_{\rr}}

\newcommand{\Man}{\mathcal{M}an}
\newcommand{\Manact}{\mathcal{M}an^{\circlearrowleft}}
\newcommand{\Don}{\mathcal{D}on}

\newcommand{\Symp}{\mathcal{S}ymp}

\newcommand{\Ham}{\mathcal{H}am}
\newcommand{\Hamhat}{\widehat{\mathcal{H}am}}

\newcommand{\Rmod}{R\text{-}mod}
\newcommand{\Rumod}{(R,u)\text{-}mod}
\newcommand{\Rdmod}{(R,d)\text{-}mod}
\newcommand{\Rudmod}{(R,ud)\text{-}mod}
\newcommand{\Rudubimod}{(R,ud,u)\text{-}bimod}
\newcommand{\Ruddbimod}{(R,ud,d)\text{-}bimod}

\newcommand{\fBialg}{f\text{-}\mathcal{B}ialg}
\newcommand{\fCoalg}{f\text{-}\mathcal{C}oalg}
\newcommand{\fAlg}{f\text{-}\mathcal{A}lg}
\newcommand{\uBimod}{u\text{-}\mathcal{B}imod}
\newcommand{\uBimodact}{(u\text{-}\mathcal{B}imod)^{\circlearrowleft}}
\newcommand{\dBimodact}{(d\text{-}\mathcal{B}imod)^{\circlearrowleft}}
\newcommand{\PreuBimodact}{(\Pcal re\text{-}u\text{-}\mathcal{B}imod)^{\circlearrowleft}}

\newcommand{\PreBialg}{\mathcal{P}re\mathcal{B}ialg}
\newcommand{\PreAlg}{\mathcal{P}re\mathcal{A}lg}
\newcommand{\PreCoalg}{\mathcal{P}re\mathcal{C}oalg}

\newcommand{\Ch}{\mathcal{C}h}

\newcommand{\dCat}{d\text{-}\mathcal{C}at}

\newcommand{\Cob}{\mathcal{C}ob}

\newcommand{\WW}{Wehrheim and Woodward}

\title[The Morse complex is an $\infty$-functor]{The Morse complex is an $\infty$-functor}

\author{Guillem Cazassus}
\address{The Chinese University of Hong Kong}
\email{gcazassus@proton.me}

\thanks{GC was funded by EPSRC grant reference EP/T012749/1 and the Simons Collaboration grant no. 994320.}
\begin{document}

\begin{abstract}
We show that the Morse complex of a compact Lie monoid can be given the structure of an $f$-bialgebra, a chain-level version of bialgebras introduced in \cite{biass}; and that this assignment defines an $\infty$-functor. As a consequence, we obtain two other $\infty$-functors mapping closed smooth manifolds to their Morse complexes with their \Ainf-coalgebra structures; and closed smooth manifolds with compact Lie group actions to their Morse complexes, with a ``$u$-bimodule'' structure (a bimodule version for $f$-bialgebras).
\end{abstract}

\maketitle
\tableofcontents

\newpage

\section{Introduction}
\label{sec:intro}
In \cite{biass}, with Alexander Hock and Thibaut Mazuir we introduced the notion of $f$-bialgebras, a homotopy version of bialgebras; which we conjectured to appear in Morse and Floer theory, and to be suitable for describing the effect of a compact Lie group action on Fukaya categories. 
In this paper we focus on Morse theory.

\begin{defi}\label{def:LieMon} Let $\LieMon$ stand for the category of compact Lie monoids:
\begin{itemize}
\item  objects are closed smooth manifolds with components of possibly different dimensions, endowed with a smooth associative  product,
\item morphisms are smooth maps preserving the product.
\end{itemize}

\end{defi}

Homology groups define a functor from $\LieMon$ to the category of bialgebras. The $f$-bialgebras introduced in \cite{biass} also form a category. One might expect that Morse complexes define a functor to this category, but this is too much to ask: functoriality only holds up to homotopy.

In this paper we show that $f$-bialgebras also form an $(\infty,1)$-category $\fBialg$ (see Proposition-Definition~\ref{propdef:fBialg}). Our main result can be stated as follows:

\begin{theoremABC}
\label{th:LieMon_to_fBialg} Given some \emph{coherent choices of perturbations} (which exist), there exists an $\infty$-functor  $CM_* \colon \LieMon \to \fBialg$ that maps a compact Lie monoid (resp. a smooth morphism) to its Morse complex (resp. its Morse pushforward). 
Dually, there exists a contravariant cochain $\infty$-functor  $CM^* \colon \LieMon \to \fBialg$.
\end{theoremABC}

We will use weak Kan complexes as models for $(\infty,1)$-categories, and identify an ordinary category $\Ccal$ with its simplicial nerve $N(\Ccal)$. We will actually define the above mentioned $\infty$-functor on a larger weak Kan complex $\LieMonpert$ that contains all perturbation data, and projects to $\LieMon $ by forgetting these. We will define a ``coherent choice of perturbations'' to be a section $\LieMon \to \LieMonpert$ of the forgetful functor $\LieMonpert \to \LieMon$.

As a consequence of Theorem~\ref{th:LieMon_to_fBialg}, we obtain a statement for compact Lie group actions on closed smooth manifolds, which can be seen as a functorial version of \cite[Conj.~A]{biass}.

\begin{defi}\label{def:Manact} Let $\Manact$ stand for the category whose:

\noindent\begin{minipage}{.2\textwidth}\begin{tikzcd}
   G \ar{r}{\varphi}& G' \\
   X \ar{r}{f} \ar[,loop ,out=123,in=57,distance=2.5em]{}{} \ar[,loop ,out=-123,in=-57,distance=2.5em]{}{}& X' \ar[,loop ,out=123,in=57,distance=2.5em]{}{} \ar[,loop ,out=-123,in=-57,distance=2.5em]{}{}\\
   H \ar{r}{\psi}& H'
\end{tikzcd}
\end{minipage}
\begin{minipage}{.7\textwidth}
\begin{itemize}
\item objects consist in triples $(G,X,H)$, where $X$ is a closed smooth manifold, on which a compact Lie group $G$ (resp. $H$) acts on the left (resp. right).
\item morphisms are triples $(\varphi, f, \psi)$ as in the diagram: $\varphi,  \psi$ are group morphisms, and $f$ is a smooth map which is bi-equivariant through $(\varphi,  \psi)$.
\end{itemize}
\end{minipage}

\end{defi}
Observe that there is a functor  
\ea\label{eq:Manact_to_LieMon}
\Manact &\to \LieMon \\
(G,X,H) &\mapsto G \sqcup X \sqcup H \sqcup \left\lbrace pt \right\rbrace, \nonumber
\ea
where the product on $G \sqcup X \sqcup H \sqcup \left\lbrace pt \right\rbrace$ is given by:
\e
x\cdot y = \begin{cases} x\cdot y &\text{if defined,}\\
pt &\text{otherwise.}
\end{cases}
\e
In Definition~\ref{def:uBimodact} we will introduce another $(\infty, 1)$-category $\uBimodact$ which, loosely speaking, is to $\fBialg$ what $\Manact$ is to $\LieMon$; as well as a dual analog $\dBimodact$. We then obtain:

\begin{corABC}\label{cor:Manact_to_uBimodact} Given some coherent choices of perturbations (which exist), there exists an $\infty$-functor  $CM_*^{\circlearrowleft} \colon \Manact \to \uBimodact$. 
Dually, there exists a contravariant cochain $\infty$-functor  $CM^{*,\circlearrowleft} \colon \Manact \to \dBimodact$.
\end{corABC}

Concretely, the content of this Corollary is:
\begin{itemize}
\item to a triple $T=(G,X,H)$, corresponds a triple 
\e
(A,M,B) = (CM_*(G),CM_*(X), CM_*(H)) = CM_*(T),
\e 
where $A, B$ are $f$-bialgebras, and $M$ is an $(A,B)$ $u$-bimodule in the sense of \cite[Def.~3.13]{biass}, i.e. a bimodule w.r.t. the ascending multiplicative structure.
\item to a morphism of triples $\Phi_{01}\colon T_1 \to T_0$ corresponds a morphism of $u$-bimodules  $(\Phi_{01})_*$ in the sense of \cite[Def.~3.15]{biass},
\item these pushforwards are functorial up to homotopy:
\e
(\Phi_{01} \circ \Phi_{12})_* = (\Phi_{01})_* \circ (\Phi_{12})_* + \d H_{012} +  H_{012} \d
\e
\item such homotopies, in their turn, satisfy some coherence relations up to some higher homotopies,
\item etc... 
\end{itemize}

Forgetting one multiplicative structure, we construct two variations $\fAlg, \fCoalg$ of $\fBialg$ that can be viewed respectively as $\infty$-categories of \Ainf-algebras and \Ainf-coalgebras. With $\Man$ the category of closed smooth manifolds:

\begin{corABC}[]\label{cor:fun_Man}

Given some coherent choices of perturbations, there exists an $\infty$-functor  $CM_* \colon \Man \to \fCoalg$ that maps a closed smooth manifold (resp. a smooth map) to its Morse complex (resp. its Morse pushforward). 
Dually, there exists  a similar $\infty$-functor  $CM^* \colon \Man \to \fAlg$ corresponding to Morse cochains.

\end{corABC}

Both $\fCoalg$ and $\fAlg$ come with forgetful functors  to the $\infty$-category of chain complexes $N_{dg}(\Ch_{\Rmod})$. Therefore, in particular:
\begin{corABC}[]\label{cor:fun_Man_Kom} Morse chains and cochains provide $\infty$-functors $\Man \to N_{dg}(\Ch_{\Rmod})$.
\end{corABC}
To summarize, one has the following diagram:
\begin{center}
\begin{tikzcd}[column sep=tiny]
\LieMon  \ar[rrr, bend left=20, "CM_*"'] \arrow[rrr, bend right=15, "CM^*"] & \
 & 
  & \fBialg 
   & \\
\Manact \ar[rr, "CM_*^{\circlearrowleft}"]\ar[u]
&
 & \uBimodact  \ar[from=d, crossing over]\ar[ur]
  & 
   & \dBimodact , \ar[from=llll, crossing over, bend right=15, "CM^{*,\circlearrowleft}"'] \ar[ul]\\
\Man \ar[rr, "CM_*"] \ar[u]
&
 & \fCoalg  \ar[dr]
  & 
   & \fAlg ,\ar[dl] \ar[u] \ar[from=llll, crossing over, bend right=15, "CM^*"'] \\
&
 & 
   & N_{dg}(\Ch_{\Rmod}) 
    &
\end{tikzcd}
\end{center}

Finally, we believe that the functor from Corollary~\ref{cor:Manact_to_uBimodact} admits a categorification, and that such categorification should be part of a Morse theory counterpart to Donaldson-Floer theory and extended TFTs. Indeed, the objects of $\Manact$ correspond to the 1-morphisms  of a smooth (i.e. non-symplectic) version of a Moore-Tachikawa (partial) category $\Lier$ introduced in \cite{ham}. We believe that the functor should be categorified to a (partial) $(\infty,2)$-functor for this category, and should also have a Fukaya category $(\infty,2)$-counterpart. We will discuss this in more detail in Section~\ref{sec:rel_other_appr}, and refer to \cite{ham} and \cite[Sec.~1.2]{biass} for more context.

\subsection{Related work}\label{ssec:rel_work}

Grafted flowtrees also appear in Mazuir's work \cite{Mazuir1,Mazuir2}, who attempted to define an $\infty$-category of \Ainf algebras \cite[Sec.~2.4]{Mazuir2}. A construction of such an $\infty$-category is also given in \cite{OhTanaka_quotients_localizations}, who apply this construction to construct group actions on wrapped Fukaya categories \cite{OhTanaka_actions}, following \cite{Teleman_icm}.

Related constructions appear in \cite{HLSfinite} \cite{HLSlie} for defining equivariant Lagrangian Floer homology. More recently, \cite{PascaleffSibilla} suggest construction of $\infty$-category versions of ``Fukaya-type'' categories, via shifted symplectic geometry.

On the combinatorics side, \cite{PilaudPoliakova_Hochschild} introduce polytopes corresponding to ``$m$-painted $n$-trees'', which might correspond to our $J(m)_n$.

\subsection{Organization of the paper}\label{ssec:org_paper}

We first gather some background notions about Morse theory, infinity categories and $f$-bialgebras  in Section~\ref{sec:background}. We then give an informal outline of the constructions in Section~\ref{sec:informal_outline}. To convey the main ideas in the simplest setting, we start by outlining the proof of Corollary~\ref{cor:fun_Man_Kom}, and then explain how to elaborate on this proof and prove our other results.

In Section~\ref{sec:higher_bimultipl} we introduce the ``$n$-grafted bimultiplihedron'': a moduli space of abstract $n$-grafted biforests, which generalizes the constructions in \cite{biass}, and construct their ``partial compactification''. These will play the role of parameter spaces to construct our structures, analogously to the associahedron for the Fukaya category.

In Section~\ref{sec:three_wKan} we construct the target $\infty$-categories $\fBialg$, $\uBimodact$, $\dBimodact$  $\fAlg$ and $\fCoalg$, as subcategories of $dg$ nerves of chains of ``foresty module categories'' that we introduce.

In Section~\ref{sec:tautol_fam_pert} we construct tautological families of grafted graphs over the $n$-grafted bimultiplihedron, which allows us to set the perturbation scheme we will use.

In Section~\ref{sec:Moduli_spaces_Lie} we construct the moduli spaces of $n$-grafted flowgraphs, and use them to construct the functor of Theorem~\ref{th:LieMon_to_fBialg}. We then explain how to prove Corollaries~\ref{cor:Manact_to_uBimodact} and \ref{cor:fun_Man}.

In Section~\ref{sec:rel_other_appr} we explain in more detail the $(\infty,2)$-functor outlined above, relations with extended TFTs, and with constructions from \WW, Ma'u and Bottman.

\begin{acknow}We thank Dominic Joyce, Guillaume Laplante-Anfossi, Thibaut Mazuir, Alex Ritter and  Bruno Vallette for helpful conversations.
\end{acknow}

\section{Some background}
\label{sec:background}
Here we quickly review some standard material, and set our notations and conventions.

\subsection{Morse theory}
\label{ssec:Morse_background}

We briefly recall the Morse complex and the Morse pushforwards of smooth maps, mostly borrowing material from \cite{equiv}. We refer to \cite{AudinDamian} and \cite[Section~2.8]{KMbook} for more details. 
Let $X$ be a compact smooth manifold of dimension $n$, and $f\colon X\to \rr$ a Morse function. Each critical point $x$ has a Morse index $\ind(x)$. Denote respectively the set of critical points and index $k$ critical points by $\Crit(f)$ and $\Crit_k(f)$.

\begin{defi}\label{def:pseudo_grad} A \emph{pseudo-gradient} for $f$ is a vector field $V\in\mathfrak{X}(X)$ on $X$ such that for all $x\in X\setminus \Crit(f)$, $d_xf. V<0$; and such that in a Morse chart near a critical point, $V$ is the negative gradient of $f$ for the standard metric on $\rr^n$. Denote by $\mathfrak{X}(X,f)\subset\mathfrak{X}(X)$ the space of pseudo-gradients for $f$. This is a convex (hence contractible) space.
\end{defi}

We will usually fix a pseudo-gradient on $X$ and leave it implicit in the notations: we will denote $\phi_X^t$ the time $t$ flow of $V$.

\begin{defi}\label{def:stable_unstable_mfd}Let $x\in \Crit_k(f)$   and  $V\in \mathfrak{X}(X,f)$. Define the \emph{stable} (resp.  \emph{unstable}) submanifold of $x$ by:
\ea
S_x &= \left\lbrace y\in X :\lim_{t\to +\infty} \phi_X^t(y)=x \right\rbrace,\\
U_x &= \left\lbrace y\in X :\lim_{t\to -\infty} \phi_X^t(y)=x \right\rbrace .
\ea
The subsets $S_x$ and $U_x$ are smooth (non-proper) submanifolds diffeomorphic respectively to $\rr^k$ and $\rr^{n-k}$.  
\end{defi}

\begin{defi}\label{def:Palais_Smale_cond} A pseudo-gradient $V$ is \emph{Palais-Smale} if for any pair $x,y$ of critical points, $U_x$ intersects $S_y$ transversally.
\end{defi}

Let $R$ be a ring (say $R= \zz$). If $x\in \Crit f$, let $o,o'$ be the two orientations of $U_x$. As in \cite{seidel2008fukaya}, let the normalization of $x$ be the line
\e
\abs{x}_R = (R o \oplus R o' )/o+o'.
\e
If $V$ is Palais-Smale, $\Mcal_{\partial}(x,y) = (U_x\cap S_y)/\rr$, where $\rr$ acts on $U_x\cap S_y$ by the flow of $V$. It is \emph{oriented relatively to $x$ and $y$}, in the sense that orientations $o_x, o_y$ on $U_x, U_y$ canonically induce an orientation on $\Mcal_{\partial}(x,y)$, and reversing either  $o_x$ or $o_y$ reverses the orientation on $\Mcal_{\partial}(x,y)$. If $\ind (y) = \ind (x) -1$, a choice of $o_x, o_y$ gives an integer $\# \Mcal_{\partial}(x,y)_{o_x, o_y}\in \zz$, and a map
\e
\abs{\Mcal_{\partial}(x,y)}_R \colon \abs{x}_R \to \abs{y}_R 
\e
independent on $o_x, o_y$. 
\begin{defi}\label{def:Morse_complex} Assume $V\in\mathfrak{X}(X,f)$  is Palais-Smale. Define the \emph{Morse complex} 
\[
CM_*(X,f),\ \partial\colon CM_*(X,f)\to CM_{*-1}(X,f) 
\] 
by
\ea
CM_k(X,f) &= \bigoplus_{x\in \Crit_k(f)} \abs{x}_R\\
\partial  &= \sum_{\ind (y) = \ind (x) -1}\abs{\Mcal_{\partial}(x,y)}_R.
\ea
One has $\partial^2 = 0$ and  $HM_*(X,f) = H_*(X, R)$. Intuitively, $x\in  \Crit_k(f)$ corresponds to a $k$-chain obtained by triangulating $U_x$.
\end{defi}

Let $F\colon X \to Y$ be a differentiable map between two smooth compact manifolds. 
Endow $X$ and $Y$ with two Morse functions $f\colon X\to \rr$, $g\colon Y\to \rr$, and pseudo-gradients $V, W$. Let $x\in \Crit_k(f)$ be a generator of $CM_k(X,f)$ (say $k\geq 1$ for the following discussion). Heuristically, $x$ corresponds to the $k$-chain of its unstable manifold $U_x$, therefore its image $F_*x$ should correspond to $F(U_x)$, which is a priori unrelated to $g$. Apply the flow of $W$ to it: for $t$ large enough, most points will fall down to local minimums, except those points lying in a stable manifold $S_y$ of a critical point $y$ of index $l\geq 1$. If $k=l$, then a small neighborhood of those points will concentrate to $U_y$, which now corresponds to a generator of $CM(Y,g)$. Loosely speaking, we want to replace $F(U_x)$ by ``$\lim_{t\to +\infty}\phi^t_{Y}( F(U_x))$''. This defines the right homology class, but breaks (ordinary) functoriality at the chain level. We will show that functoriality can be restored in the $\infty$-category world.

Therefore the previous discussion motivates the following definition. Assume that $f$, $g$ and two pseudo-gradients $V, W$ are chosen so that, for any critical points $x$ and $y$ of $f$ and $g$ respectively, the graph $\Gamma(F)$ intersects $U_x\times S_y$ transversely in $X\times Y$. Then $\Mcal(F;x,y) = \Gamma (F) \cap (U_x \times S_y)\simeq  U_x \cap F^{-1} (S_y)$ is oriented relatively to $x$ and $y$, and of dimension $\ind(x) - \ind(y)$. 
Define then 
\ea \nonumber
F_* &\colon CM_*(X,f)\to CM_*(Y,g)\text{ by} \\
F_*  &=  \sum_{\ind (y) = \ind (x) }\abs{\Mcal(F;x,y)}_R. \label{eq:Morse_pushforwards}
\ea
since $\Mcal(F;x,y)$ corresponds to those points whose neighborhoods concentrate to $U_y$. 
This map induces the actual pushforward in homology \cite[prop.~2.8.2]{KMbook}.

It is convenient to think of $\Mcal(F;x,y)$ as a moduli space of \emph{grafted flow lines} from $x$ to $y$: By this we mean a pair of flow lines $(\gamma_-,\gamma_+)$, with
\begin{align*}
\gamma_- &\colon \rr_- \to M,\ \gamma_-'(t) =  V(\gamma_-(t)), \\
\gamma_+ &\colon \rr_+ \to N,\ \gamma_+'(t) =  W(\gamma_+(t))
\end{align*}
such that 
\begin{align*}
&\lim_{t\to -\infty}{\gamma_-(t)} = x, \\ 
&\lim_{t\to +\infty}{\gamma_+(t)} = y, \\
&F\left(\gamma_-(0)\right) = \gamma_+(0). 
\end{align*}
The identification with $\Mcal(F;x,y)$ is given by:
\e
(\gamma_-,\gamma_+)\mapsto (\gamma_-(0),\gamma_+(0)).
\e

\subsection{Infinity categories}
\label{ssec:infty_cat_background}

We quickly recall the weak Kan complex model of $(\infty,1)$-categories, and refer to \cite{GrothShortCourse} for more details.

A \emph{simplicial set} is a functor $\Delta^{\rm op} \to \Scal et$, with $\Delta$ standing for the category of finite ordinals $[n] = (0< \cdots < n)$ and order-preserving maps. The category of simplicial sets is denoted $s\Scal et$. 
Any order-preserving map induces maps between simplices, in particular one has (co)faces and (co)degeneration maps.

A simplicial set is a \emph{weak Kan complex} if every inner horn can be filled. An \emph{inner horn} is a collection of $n$ $(n-1)$-simplices arranged as the $i$-th  horn of the standard $n$-simplex $\Lambda_i^n = \partial \Delta^n \setminus \partial_i \Delta^n$, $1\leq i \leq n-1$.

Any category  $\Ccal$ has a nerve $N(\Ccal) \in s\Scal et$, where an $n$-simplex consists of objects $x_0, \ldots ,x_n$ and maps $f_{ij} = f_{i(i+1)} \circ \cdots \circ f_{(j-1)j} \colon x_j \to x_i$, for $i<j$.

A variant of this construction exists for $dg$ categories. It first appeared in \cite[Def.~A2.1]{HinichSchechtman} under the name of Sugarawa $n$-simplices, and was extended to $dg$ categories and further studied by Lurie \cite[Constr.~1.3.1.6]{LurieHA}.

If $\sigma = [\sigma_0, \ldots , \sigma_m]$, 
we will write $m = \dim \sigma$, and for $0\leq i \leq m$, we denote 
\ea
\partial_i \sigma &= [\sigma_0, \ldots , \widehat{\sigma}_i ,  \ldots , \sigma_m] , \\
\sigma_{\leq i} &= [\sigma_0, \ldots , \sigma_i] , \\
\sigma_{\geq i}  &= [\sigma_i, \ldots , \sigma_m] .
\ea

Let $\Ccal$ be a $dg$ category, its $dg$ nerve $N_{dg}(\Ccal)$ is a weak Kan complex such that, for $n\geq 0$, its $n$-simplices are pairs $(\lbrace A_i \rbrace_{0\leq i \leq n} , \lbrace\varphi_{\sigma} \rbrace_{\sigma\subset [n]})$, with $A_i$ objects of $\Ccal$, and for $\sigma = [\sigma_0, \ldots , \sigma_m]\subset [n]$, $m\geq 1$, $\varphi_{\sigma} \colon A_{\sigma_m} \to A_{\sigma_0}$ is of degree $m-1$, and satisfy the coherence relations:
\e\label{eq:cohrel_dg_nerve}
\partial (\varphi_{\sigma}) = \sum_{i=1}^{m-1} (-1)^{i} ( \varphi_{\partial_i \sigma} - \varphi_{\sigma_{\leq i}}  \circ \varphi_{\sigma_{\geq i}} ).
\e

In particular, this construction applies to the $dg$ category of chain complexes $\Ch_{\Acal}$ over a graded linear category $\Acal$.

\subsection{$f$-bialgebras}
\label{ssec:f_bialg_background}

We very briefly remind some notions and notations from \cite{biass}, and refer to this paper for more detail.

A (rooted ribbon) \emph{forest} $\varphi
$ consists in:
\begin{itemize}
\item Ordered finite sets of leaves $\Leaves(\varphi)$ and roots $\Roots(\varphi)$,
\item Finite sets of edges $\Edges(\varphi)$ and vertices  $\Vert(\varphi)$,
\item Source and target maps 
\ea
s &\colon \Edges(\varphi) \to \Vert(\varphi) \cup \Leaves(\varphi),\nonumber\\
t &\colon \Edges(\varphi) \to \Vert(\varphi) \cup \Roots(\varphi),\nonumber
\ea
\end{itemize}
satisfying some conditions.

An ascending (resp. descending) forest $U=(\varphi,h)$ (resp. $D=(\varphi,h)$) is a forest equipped with a function $h\colon \Vert(\varphi) \to \rr$ that decreases (resp. increases) along edges. A biforest is a pair $B = (U,D)$.

We use multi-indices $\kund = (k_1, \ldots , k_a)$ (resp.  $\lund = (l_1, \ldots , l_b)$) to denote the type of an ascending (resp. descending) forests. That is, the ascending forest has $a$ trees, and the i-th tree has $k_i \geq 1$ leaves.

We say that $\varphi$ is \emph{vertical} (or trivial) if it has no vertices, i.e. $\kund = (1, \ldots , 1 )$, i.e. all edges are infinite. 
We say that $\varphi$ is \emph{almost vertical} if it is not vertical, is trivalent and has no internal edges, i.e. $\kund$ has only 1 and 2 as coefficients (and at least one 2).

The moduli spaces of ascending forests, descending forests, and biforests of a given type are respectively denoted $U^{\kund}$, $D_{\lund}$ and $J^{\kund}_{\lund} = U^{\kund} \times D_{\lund}$.

We will use the following notations associated to a multi-index $\kund$:
\begin{itemize}
\item $\abs{\kund} = k_1 + \cdots + k_a$ corresponds to the number of leaves of forests,
\item $n({\kund})=a$ the number of trees (i.e. the number of roots),
\item $\tilde{\kund}$ the multi-index obtained by removing all the entries in $\kund$ equal to 1,
\item $\tilde{a} = \tilde{n}({\kund}) = n(\tilde{\kund})$ the number of nonvertical trees,
\item $v(\kund) = \abs{\kund} - n({\kund})$ the generic number of vertices (see below),
\item we will denote $\One_a= (1, \ldots,1)$ the multi-index with $a$ entries equal to 1. 
\end{itemize}
To a multi-index $\kund$ we associate the set
\e
\Vert(\kund) = \left\lbrace 1, \ldots , \abs{\kund} \right\rbrace \setminus \left\lbrace k_1, k_1 + k_2, \ldots , \abs{\kund} \right\rbrace ,
\e
which has cardinality $v(\kund) = \abs{\kund} - n({\kund})$. One has canonical identifications
\e\label{eq:identif_bimult}
i^{\kund}_{\lund} \colon J^{\kund}_{\lund} \to \rr^{\Vert(\kund) \cup \Vert(\lund)} \simeq \rr^{v(\kund)} \oplus  \rr^{v(\lund)},
\e
which permit to orient $J^{\kund}_{\lund}$.

If  $U = (\varphi, h)$ is either an ascending or descending forest we associate the following spaces

\ea
\abs{U} &= \coprod_{e\in \Edges(\varphi)}I_e, \\
\langle U \rangle &= \abs{U}_{/_\sim}
\ea
where $I_e \subset \rr$ is either $[h(v_2), h(v_1)]$ or  $[h(v_1), h(v_2)]$, if $e\colon v_1\to v_2$ (when $h(v_i) = \pm \infty$ the bound is excluded). In $\langle U \rangle$ the equivalence relation we quotient out by is given by  identifying endpoints corresponding to a given vertex.

We also still denote $h\colon \abs{U} \to \rr$ and $h\colon \langle U \rangle \to \rr$ the obvious ``vertical'' projections. The graph associated to the biforest, which we refer to as the \emph{Henriques intersection} (see \cite[Fig.~3]{biass}), is defined to be the fibered product
\e
\Gamma(U, D) = \langle U \rangle\times_{\rr} \langle D \rangle,
\e
with induced height function. 
We let $c(\kund,\lund)$ stand for the dimension of the symmetry group preserving $\Gamma(U, D)$:
\e
c(\kund,\lund)  = \begin{cases} 
\tilde{a} &\text{ if }\tilde{b}=0,\\
\tilde{b} &\text{ if }\tilde{a}=0,\\
1&\text{ otherwise.}
\end{cases}
\e

Let $\kund^0 , \kund^1$ be two multi-indices such that $\abs{\kund^0} = n(\kund^1)$. Let ``\emph{$\kund^1$ glued on top of $\kund^0$}'' be the multi-index defined as:
\e
\kund^1 \sharp \kund^0 = ( k^1_1 + \cdots + k^1_{k^0_1} , k^1_{k^0_1 +1} + \cdots + k^1_{k^0_1 + k^0_2}, \cdots , k^1_{k^0_1 +\cdots + k^0_{a^0-1} +1} + \cdots + k^1_{a^1}).
\e
If $U^0$ and $U^1$ are ascending forests of type $\kund^0$ and $\kund^1$, this corresponds to gluing $U^1$ on top of $U^0$.

If $\kund = \kund^1 \sharp \kund^0$ and $\lund = \lund^0 \sharp \lund^1$, one has:
\e
J^{\kund}_{\lund} \simeq J^{\kund^0}_{\lund^0} \oplus J^{\kund^1}_{\lund^1},
\e
and from \cite[Lem.~2.1]{biass}, this isomorphism affects orientation by
\e
(-1)^{\heartsuit^{\kund^1}_{\kund^0} + \heartsuit^{\lund^0}_{\lund^1} + v(\lund^0) ( v(\kund^1) + v(\lund^1) )},
\e
where
\ea
\heartsuit^{\kund^1}_{\kund^0} &= \sum_{h = 1}^{a^1} (k^1_{h} -1) v_{\geq h}(\kund^0)\text{, with} \\
v_{\geq h}(\kund) &= \mathrm{Card} \left(\Vert (\kund) \cap [ h,+\infty ) \right) \\
&= (k_i - i') + (k_{i+1} - 1) + \cdots +(k_{a} - 1)\text{, if} \nonumber\\
h&= k_1 + \cdots + k_{i-1} + i'. \nonumber
\ea
The quantity $\heartsuit^{\kund^1}_{\kund^0} \in \Z{2}$ is the signature of the permutation of $\left\lbrace 1, \ldots , v(\kund) \right\rbrace$ corresponding to the identification 
$\Vert (\kund) \simeq \Vert (\kund^0) \cup \Vert (\kund^1)$.

Operations associated to biforests have inputs and outputs lying on a rectangular grid, as in \cite[Fig.~10]{biass}. These correspond to a ``rectangular tensor product'' chain complex $(A^b)^a = (A^{\otimes b})^{\otimes a}$. When splitting $a=a_1+ \cdots +a_k$ we use the obvious identification
\e\label{eq:iso_split_a}
(A^b)^a \simeq (A^b)^{a_1} \otimes \cdots \otimes (A^b)^{a_k}.
\e
However, when splitting $b=b_1+ \cdots +b_k$, the identification 
\e\label{eq:iso_split_b}
(A^b)^a \simeq (A^{b_1})^{a} \otimes \cdots \otimes (A^{b_k})^{a}
\e
is given by first exchanging $b$ and $a$ (with the Koszul sign convention), and then applying (\ref{eq:iso_split_a}).

If $a=a_1+a_2+a_3$, we identify 
\e
(A^b)^a\simeq (A^b)^{a_1 + a_3} \otimes (A^b)^{a_2} 
\e
via $\xund_1 \otimes \xund_2 \otimes \xund_3 \mapsto (-1)^{\abs{\xund_2}\abs{\xund_3}} (\xund_1 \otimes \xund_3 ) \otimes \xund_2$ Likewise, if $b=b_1+b_2+b_3$, we will identify 
\e
(A^b)^a\simeq (A^{b_1 + b_3})^a \otimes (A^{a_2})^a 
\e 
similarly, after exchanging $b$ and $a$.

Given maps $P_i\colon (A^b)^{a_i} \to (A^d)^{c_i}$ and  $P\colon (A^b)^{a} \to (A^d)^{c}$, with $a = a_1+ \cdots + a_k$ and $c = c_1+ \cdots + c_k$, whenever we write $P \simeq P_1 \otimes \cdots \otimes P_k$, we mean that the following diagram commutes, where the vertical maps are the identifications (\ref{eq:iso_split_a}):
\e
\xymatrixcolsep{.8in}
\xymatrix{ (A^b)^{a} \ar[d]\ar[r]^P & (A^d)^{c} \ar[d]\\
(A^b)^{a_1} \otimes \cdots \otimes (A^b)^{a_k} \ar[r]^{P_1 \otimes \cdots \otimes P_k} & (A^d)^{c_1} \otimes \cdots \otimes (A^d)^{c_k} 
}
\e
Likewise for maps $P_i\colon (A^{b_i})^{a} \to (A^{d_i})^{c}$ and  $P\colon (A^b)^{a} \to (A^d)^{c}$, with $b = b_1+ \cdots + b_k$, $d = d_1+ \cdots + d_k$, and the identifications (\ref{eq:iso_split_b}).

We defined an $f$-bialgebra $(A, \alpha)$ as a collection of operations 
\e
{\alpha}^{\kund}_{\lund}\colon (A^b)^{\abs{\kund}} \to (A^{\abs{\lund}})^{a},
\e
satisfying some coherence and simplification relations, which we will recall in Section~\ref{sec:three_wKan}. 
If  $(A, \alpha)$  and $(B, \beta)$  are two $f$-bialgebra, we define a morphism of $f$-bialgebras  $f\colon A \to B$ as a collection of operations:
\e
f^{\kund}_{\lund} \colon  (A^b)^{\abs{\kund}} \to (B^{\abs{\lund}})^{a}
\e
satisfying coherence and simplification relations that we will also recall in Section~\ref{sec:three_wKan}.

We also define an \emph{ascending $(A,B)$-bimodule} (or $(A,B)$ $u$-bimodule) $(M,\mu)$ to be a family of operations
\e\label{eq:u_bimod_operations}
\mu^{(\kund^l |\epsilon | \kund^r)}_{\lund}\colon (A  ^{b} )^{\abs{\kund^l}} \otimes (M  ^{b} )^{\epsilon} \otimes (  B^{b} )^{\abs{\kund^r}} \to (A  ^{\abs{\lund}} )^{a^l} \otimes (M  ^{\abs{\lund}} )^{\epsilon} \otimes (  B^{\abs{\lund}} )^{a^r},
\e
satisfying relations analogous to $f$-bialgebras. In the above, $\epsilon=0$ or $1$. When $\epsilon=0$, $\kund = (\kund^l |0 | \kund^r)$ corresponds to separating an ascending biforest to a left and a right subforest. Either $\kund^l$ or $\kund^r$ are allowed to be empty, but not at the same time. When $\epsilon=1$, $\kund = (\kund^l |1 | \kund^r)$ corresponds to choosing a leaf of $\kund$, and separate the corresponding forest to a left, central, and a right part \cite[Fig.~11]{biass}.

Finally, given an $(A,B)$ $u$-bimodule $(M,\mu)$, an $(A',B')$ $u$-bimodule $(M',\mu)$, and morphisms of $f$-bialgebras $f\colon A \to A'$, $g\colon B \to B'$, we define an $(f,g)$-equivariant morphism $\Psi\colon M \to M'$ as a collection of morphisms
\e\label{eq:mph_u_bimod_operations}
\Psi^{(\kund^l |\epsilon | \kund^r)}_{\lund}\colon (A  ^{b} )^{\abs{\kund^l}} \otimes (M  ^{b} )^{\epsilon} \otimes (  B^{b} )^{\abs{\kund^r}} \to ({A'}^{\abs{\lund}} )^{a^l} \otimes ({M'}  ^{\abs{\lund}} )^{\epsilon} \otimes (  {B'}^{\abs{\lund}} )^{a^r},
\e
satisfying relations analogous to morphisms of $f$-bialgebras.

\section{Informal outline}
\label{sec:informal_outline}

As mentioned in Section~\ref{ssec:Morse_background}, Morse pushforwards are not functorial in the ordinary sense: functoriality holds up to homotopy. That is, if $\phi_{01}\colon X_1 \to X_0$ and $\phi_{12}\colon X_2 \to X_1$ are smooth maps between closed manifolds, then 
\e
\partial (\varphi_{012}) =  (\phi_{01})_* \circ (\phi_{12})_* - (\phi_{01} \circ \phi_{12})_*,
\e
where $\varphi_{012}\colon CM(X_2) \to CM(X_0)$ is defined using a parametrized moduli space, and is part of a larger family that constitutes a simplicial map, as we construct below.

Consider a sequence of closed smooth manifolds $X_0$, ..., $X_n$. On each $X_i$, fix a Morse function $f_i$, a Palais-Smale pseudo-gradient $V_i$; and let $\phi_{X_i}^{t}$ stand for its time $t$ flow. Let also $\phi_{i(i+1)}$ be smooth maps between them, as in the diagram below ($\Phi_{\Lund}$ will be introduced later).
\begin{center}
\begin{tikzcd}
 X_n  \ar[r, "\phi_{(n-1)n}"] \ar[rrrr, start anchor={[yshift=1ex]}, bend left=30, "\Phi_{\Lund}"] 
& X_{n-1} \ar[r, "\phi_{(n-2)(n-1)}"] \ar[,loop ,out=-123,in=-57,distance=2.5em]{}{} 
 & \cdots \ar[r, "\phi_{12}"]
  & X_{1} \ar[r, "\phi_{01}"] \ar[,loop ,out=-123,in=-57,distance=2.5em]{}{} 
   & X_{0} .\\ 
&\phi_{X_{n-1}}^{L_{n-1}} 
 & 
  &\phi_{X_1}^{L_1} 
   &  
\end{tikzcd}
\end{center}

Let us denote $I = [0, +\infty)$ and $\overline{I} = [0, +\infty]$. 

\begin{defi}\label{def:n_grafted_line} An \emph{$n$-grafted line} from $x\in X_n$ to $y\in X_0$ of lengths $\Lund = (L_1, \ldots , L_{n-1}) \in I^{n-1}$ is a family of flowlines $\gamund =(\gamma_0, \gamma_1, \ldots, \gamma_n)$ as below (see Figure~\ref{fig:rectangular_box}).

The lengths $\Lund$ determine a sequence $h_{01}< \cdots < h_{(n-1)n}$ of grafting heights:
\e\label{eq:grafting_heights}
h_{01} = 0,\ h_{12} = L_1,\ h_{23} = L_1+ L_2,\ \ldots , h_{(n-1)n} = L_1+ \cdots + L_{n-1}.
\e
We call $h(\Lund) = h_{(n-1)n}$ the total height of $\Lund$. The flowlines are:
\ea
\gamma_n &\colon (-\infty, - h_{(n-1)n}] \to X_n , \nonumber\\
\gamma_{n-1} &\colon [ - h_{(n-1)n}, - h_{(n-2)(n-1)}] \to X_{n-1} , \nonumber\\
\vdots & \nonumber\\
\gamma_{1} &\colon [- h_{12},- h_{01}] \to X_{1} , \nonumber\\
\gamma_0 &\colon [- h_{01}, +\infty)  \to X_0 ,\nonumber 
\ea
and are required to satisfy the grafting conditions, for $0\leq i <n$:
\e
\gamma_i (-h_{i(i+1)}) = \phi_{i(i+1)}( \gamma_{i+1} (-h_{i(i+1)})),
\e
and the limit conditions:
\e
  \lim_{-\infty} \gamund = x ,\ \ \  \lim_{+\infty} \gamund =  y   ,
\e
with $\lim_{-\infty} \gamund = \lim_{-\infty} \gamma_{0}$ and $\lim_{+\infty} \gamund = \lim_{+\infty} \gamma_{n}$.
\end{defi}

\begin{remark}
Since we consider negative gradient flows, we parametrize the domains by $t = -h$. 
\end{remark}

\begin{figure}[!h]
    \centering
    \def\svgwidth{.80\textwidth}
    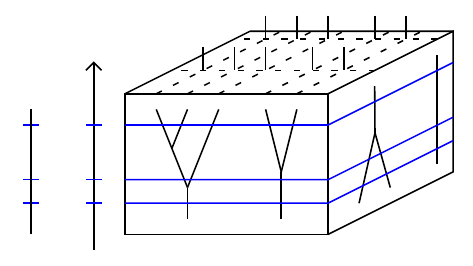
      \caption{On the left, a $3$-grafted flowline. On the right, a $3$-grafted biforest. Inside the box lives a $3$-grafted intersection graph.}
      \label{fig:rectangular_box}
\end{figure}
 
One can form the moduli space
\e
\Mcal_{[n]}(x;y) = \left\lbrace (\Lund, \gamund)\ \left| \gamund\text{ is an $n$-grafted line of lengths $\Lund$ from $x$ to $y$}  \right. \right\rbrace,
\e
which (if one perturbs the pseudogradients appropriately) is generically transversely cut out, and has dimension $\ind(x) - \ind(y) +n-1$. When of dimensions zero, it is an oriented finite set, which can be counted to define 
\e
\varphi_{0, \ldots,n} \colon A_n \to A_0,
\e
with $A_i = CM_*(X_i)$. When of dimension one, it compactifies to a compact oriented 1-dimensional manifold with boundary, and the count of its boundary points gives the coherence relation, as we shall see.

To orient these moduli spaces, it is helpful to notice that $\varphi_{0, \ldots, n}$ can be seen as a pushforward, as we now explain. 
Indeed, given $\Lund = (L_1, \ldots , L_{n-1}) \in I^{n-1}$, let 
\e\label{eq:Phi_Lund_grafted_line}
\Phi_{\Lund} := \phi_{01} \circ \phi^{L_{1}}_{X_{1}} \circ \phi_{12} \circ \cdots \circ \phi^{L_{n-1}}_{X_{n-1}}\circ \phi_{(n-1)n}.
\e
As a single map, we get
\e\label{eq:Phi_grafted_line}
\Phi\colon I^{n-1} \times X_{n} \to X_0.
\e
Equip $I^{n-1}$ with a Morse function $f_I$ having a single critical point $o$ which is in its interior, and is a local maximum. Let $V_I$ be a pseudo-gradient. Orient the unstable manifold $U_o \simeq I^{n-1}$ with the standard orientation of $I^{n-1}$. Then $\widehat{X}_n = I^{n-1} \times X_{n}$ is equipped with $f_{\widehat{X}_n} = f_{I} + f_{{X}_n}$ and  $V_{\widehat{X}_n} = V_{I} + V_{{X}_n}$, and
\e
CM_*(I^{n-1} \times X_{n}) \simeq A_n[-(n-1)].
\e
Under this identification, one has $\varphi_{0, \ldots,n}= \Phi_*$.

Let us pretend that $\Phi$ extends smoothly to the closure $\widetilde{X}_n = \overline{I}^{n-1} \times X_{n}$. 
This is far from being true: the limit of the flow at infinity is a discontinuous map. Nevertheless, it gives the right intuition and the right signs.

Observe that for the above choice of Morse functions, $CM_*(\widetilde{X}_n)$ models the complex of $\widetilde{X}_n$ relative to its boundary. 
As a general fact, we have that 
\e\label{eq:bdary_pullback}
\partial (\Phi_*) = (\Phi_{|\partial \widetilde{X}_n})_*, 
\e
where 
\ea
\partial (\Phi_*) &= \partial_{X_0} \circ \Phi_* -  \Phi_* \circ \partial_{\widetilde{X}_n} \\
&\simeq \partial_{X_0} \circ \Phi_* - ( -1)^{n-1}  \Phi_* \circ \partial_{X_n} = \partial (\varphi_{0, \ldots,n}), 
\ea
and $\Phi_{|\partial \widetilde{X}_n}$ stands for the restriction of $\Phi$ to 
\e\label{eq:bdary_cube_times_Mn}
\partial \widetilde{X}_n = \left( \bigcup_{i=1}^{n-1} (-1)^{i} ( F_i^{0} \cup  - F_i^{+\infty} ) \right) \times X_n ,
\e 
where $F_i^{\alpha} = \left\lbrace x_i = \alpha \right\rbrace$ are the faces of the cube $\overline{I}^{n-1}$. From (\ref{eq:bdary_cube_times_Mn}), we obtain
\ea
(\Phi_{|\partial \widetilde{X}_n})_* = \sum_{i=1}^{n-1} (-1)^{i} ( \varphi_{0, \ldots, \widehat{i}, \ldots,n} - \varphi_{0, \ldots,i} \circ \varphi_{i, \ldots,n}),
\ea
and therefore (\ref{eq:bdary_pullback}) is the coherence relation (\ref{eq:cohrel_dg_nerve}) of the $\infty$-category $N_{dg}(\Ch_{\Rmod})$. This ends the informal proof of Corollary~\ref{cor:fun_Man_Kom}.

Let us now turn to Corollary~\ref{cor:fun_Man}: we now want to include the comultiplicative structure of the Morse complex. For a given manifold $X$, it is dual to the cup-product in cohomology, which is induced by the diagonal map $\Delta \colon X \to X\times X$. At the chain level, it becomes the \Ainf-algebra structure introduced by Fukaya by counting Morse flow trees \cite{Fukaya_htpy,FukayaOh_zeroloop,Mescher_book}. Since we consider chains, we get a dual \Ainf-coalgebra structure.

Consider a sequence of closed smooth manifolds and smooth maps as before. 
A natural idea is to replace the vertical line in the previous construction by a vertical tree: one would get $n$-grafted trees with $l$ leaves. These would form a family 
\e
J(n)_l = I^{n-1} \times J_l,
\e
where $J_l$ stands for the interior of the multiplihedron, which parametrizes trees with one grafting level. These would allow one to define moduli spaces of grafted flow trees that one could count in order to define maps

\e
(\varphi_{0\ldots n})_{l} \colon A_n \to (A_0)^l .
\e
When trying to write down coherence relations for these maps, one faces the issue that the natural compactification of $J(n)_l$ is not ``face-coherent'', i.e. its boundary faces are not products of similar spaces $J(n')_{l'}$. This prevents to express boundary terms of the one-dimensional compactified moduli spaces as combinations of maps $(\varphi_{0\ldots n'})_{l'} $. A similar problem appeared in \cite{biass}, we refer to Section~1.1 of this paper for a more detailed discussion.  A similar solution applies in the present setting: one can enlarge the family by replacing trees by forests. Therefore, the integer $l$ becomes a multi-index $\lund = (l_1, \ldots , l_b)$; one gets spaces of grafted forests $J(n)_{\lund}$, which are now face-coherent; and induce a family of maps
\e
(\varphi_{0\ldots n})_{\lund} \colon (A_n)^b \to (A_0)^{l_1 + \cdots + l_b} ,
\e
for which one can write down coherence relations. This will give us a definition for an $\infty$-category of \Ainf -coalgebras, which we call $\fCoalg$; and an $\infty$-functor as stated.

The proof of Theorem~\ref{th:LieMon_to_fBialg} is then analogous, replacing forests by Henriques intersections of biforests, as in \cite{biass} (right of Figure~\ref{fig:rectangular_box}). One will then consider spaces of grafted biforests $J(n)^{\kund}_{\lund}$ indexed by two multi-indices $\kund = (k_1, \ldots , k_a)$ and $\lund = (l_1, \ldots , l_b)$ which will induce a family of maps
\e
(\varphi_{0\ldots n})^{\kund}_{\lund} \colon ((A_n)^b)^{k_1 + \cdots + k_a} \to ((A_0)^{l_1 + \cdots + l_b})^a .
\e
Since we can regard $J(n)_{\lund}=J(n)^{(1)}_{\lund}$ as a sub-family of $J(n)^{\kund}_{\lund}$, we will construct and (``partially'') compactify $J(n)^{\kund}_{\lund}$; and then derive from it the definitions of $\fBialg$, $\fCoalg$ and $\fAlg$. In fact, they will be subcategories of some $dg$ nerves.

In the above discussion we implicitly assumed that all moduli spaces were transversely cut out. This might not be possible if one  uses fixed pseudo-gradients on each $X_i$. 
Nevertheless, as long as the domains are stable, one can use domain-dependent perturbations of these pseudo-gradients, as is now standard in Morse and Floer theory. 
Since domains vary smoothly in families, perturbations should also depend smoothly on the domains. Therefore, we will follow a similar approach as in \cite{seidel2008fukaya}, and construct tautological families of grafted graphs over $J(n)^{\kund}_{\lund}$, which will serve as domains for the perturbations.

\section{Higher bimultiplihedron}
\label{sec:higher_bimultipl}

\begin{defi}\label{def:n_grafted_biforest} For $n\geq 1$, an \emph{$n$-grafted biforest} is a pair $(\Lund, B)$, where $\Lund \in I^{n-1}$, and $B=(U,D)$ is a biforest, i.e. $U$ (resp. $D$) is an ascending (resp. descending) forest. 
A $0$-grafted biforest is a class of biforests, modulo translation.

If $U$ is of type $\kund$ and $D$ of type $\lund$, let the type of $(\Lund, U,D)$ be $\dd = (n,\kund,\lund)$. 
Let $J_{\dd} = J(n)^{\kund}_{\lund}$ stand for the moduli space of grafted forests of type $\dd$. If $n=0$ we let, as in \cite{biass}:
\e
J(0)^{\kund}_{\lund}= K^{\kund}_{\lund} =\begin{cases} J^{\kund}_{\lund} /\rr &\text{if }c(\kund,\lund)=1, \\ \emptyset &\text{otherwise}.\end{cases}
\e
From the identification (\ref{eq:identif_bimult}), $J(n)^{\kund}_{\lund}$ is the standard corner: 
\e
J(n)^{\kund}_{\lund}= K^{\kund}_{\lund} =  I^{n-1} \times  J^{\kund}_{\lund} \simeq I^{n-1} \times  \rr^{v(\kund)} \times \rr^{v(\lund)},
\e
and we orient it as such.
\end{defi}

As in \cite{biass}, we construct a partial compactification by attaching boundary components along gluing maps.

If $\dd^0 = (n^0,\kund^0,\lund^0)$ and  $\dd^1 = (n^1,\kund^1,\lund^1)$, write when it makes sense:
\e
\dd^1 \sharp \dd^0 = (n^0 + n^1,\kund^1 \sharp \kund^0,\lund^0 \sharp \lund^1).
\e

Recall that one has a natural (unoriented) identification $J^{\kund}_{\lund}\simeq J^{\kund^0}_{\lund^0}\times J^{\kund^1}_{\lund^1}$. 
We denote $v^i = (1, \ldots, 1) \in J^{\kund^i}_{\lund^i}$ the generators of the translation actions. 
If $\dd= \dd^1 \sharp \dd^0$, let the gluing map be the linear map

\ea
g_{\dd^0, \dd^1}\colon I \times J_{\dd^0} \times J_{\dd^1} &\to J_{\dd} \\
(L, (\Lund^0, B^0),(\Lund^1, B^1)) &\mapsto (\Lund, B ),\nonumber
\ea
with, identifying the quotients as linear subspaces $J(0)^{\kund}_{\lund} \subset J^{\kund}_{\lund}$:

\begin{itemize}
\item if $n_0, n_1 \geq 1$,
\ea
\Lund &= (L_1^0, \ldots , L_{n_0 - 1}^0, L, L_1^1 \ldots, L_{n_1 - 1}^1) , \\
B &= (B^0 , B^1 + ( h(\Lund^0 )+ L  ) v^1) .
\ea
That is, when $L$ is large enough,  $B$ is obtained by gluing $B^1$, shifted upwards by $h(\Lund^0 )+ L$, on top of $B^0$.

\item if $n_0=0$ and $ n_1 \geq 1$,
\ea
\Lund &= \Lund^1 , \\
B &= (B^0 - L v^0, B^1 ) .
\ea

\item if $n_1=0$ and $ n_0 \geq 1$,
\ea
\Lund &= \Lund^0 , \\
B &= (B^0 , B^1 + L v^1) .
\ea

\item if  $n_0=n_1=0$,
\e
B = \mathrm{proj}(B^0 , B^1 + L v^1) ,
\e
with $\mathrm{proj}\colon J^{\kund}_{\lund} \to J(0)^{\kund}_{\lund}$ the projection to the quotient.
\end{itemize}

Then  we can define 
\e
\left(\overline{J}_{\dd} \right)_{\leq 1} =  \left( {J}_{\dd} \cup \bigcup_{\dd= \dd^1 \sharp \dd^0} I \times {J}_{\dd^0} \times {J}_{\dd^1} \right) /\sim ,
\e
where $\sim$ identifies points in $I \times {J}_{\dd^0} \times {J}_{\dd^1}$ with their images under the gluing map $g_{\dd^0, \dd^1}$. It has boundary:
\e
\partial \left(\overline{J}_{\dd} \right)_{\leq 1} = \bigcup_{i=1}^{n-1} \partial_i J_{\dd} \cup \bigcup_{\dd= \dd^1 \sharp \dd^0} (-1)^{\rho} {J}_{\dd^0} \times {J}_{\dd^1} , 
\e
where we let $\partial_i J_{\dd} = \left\lbrace L_i =0 \right\rbrace \subset \partial J_{\dd}$, and $\rho= \rho_{\dd^0, \dd^1} \in \Z{2}$ is the orientation of $g_{\dd^0, \dd^1}$, which we compute below. 
Observe that $\partial_i J_{\dd} \simeq (-1)^{i} J_{\partial_i \dd }$, with $\partial_i \dd =(n-1, \kund,\lund)$.

\begin{remark}\label{rem:J_over_cube}
The projection $J_{\dd} \to I^{n-1}$ extends to $\left(\overline{J}_{\dd} \right)_{\leq 1} \to \overline{I}^{n-1}$, so we think of $\left(\overline{J}_{\dd} \right)_{\leq 1}$ as lying on top of the cube, and its boundary components are or three different kinds:
\begin{itemize}
\item $L_i=0$, contributing to the vertical parts of the boundary $J_{\dd}$,
\item $L_i = +\infty$, contributing to the vertical parts of the form ${J}_{\dd^0} \times {J}_{\dd^1}$, with $n_0, n_1\geq 1$,
\item $n_0=0$ or $n_1 = 0$, contributing to horizontal boundaries.
\end{itemize}
\end{remark}

\begin{lemma}\label{prop:or_glueing_maps} With $x^i = v(\kund^i)$, $y^i = v(\lund^i)$, $\heartsuit_k^{0,1} = \heartsuit^{\kund^1}_{\kund^0}$, $\heartsuit_l^{0,1} = \heartsuit^{\lund^0}_{\lund^1}$, the orientation of the gluing map is given by:
\e\label{eq:rho}
\rho = \heartsuit_k^{0,1} + \heartsuit_l^{0,1} + y^0 (x^1 + y^1)+ n_0 + 1  + (n_1 + 1 )(x^0 + y^0)
\e

\end{lemma}

\begin{remark} When $(n_0, n_1) = (0,0)$, $(1,0)$ and $(0,1)$, we respectively get the formulas for $\rho, \rho_0$ and $\rho_1$ in \cite[Prop.~2.12]{biass}
\end{remark}

\begin{proof}

Below we will prove the formula when $n_0, n_1\geq 1$, but observe that when $n=0$, multiplying by $I^{n-1}$ is formally equivalent to quotienting by $\rr$. This agrees with our orientation conventions, so the formula remains valid when $n_0$ or $n_1=0$.
Let $J = {J}^{\kund}_{\lund}$ and $J^i = {J}^{\kund^i}_{\lund^i}$.  
The gluing map is given by composing the isomorphisms\footnote{To be precise, gluing map is composing these isomorphisms with a shear mapping, which preserve the orientation.}
\ea
J_{\dd} &= I^{n-1} \times J ,\nonumber\\
 &\simeq I^{n^0-1} \times  I \times  I^{n^1 -1} \times J^0 \times J^1 , \nonumber\\
 &\simeq  I \times I^{n^0-1} \times J^0 \times   I^{n^1 -1}  \times J^1 = I \times J_{\dd^1} \times J_{\dd^1}  .\nonumber
\ea
From \cite[Lemma~2.11]{biass}, the isomorphism on the second line affects sign by $\heartsuit_k^{0,1} + \heartsuit_l^{0,1} + y^0 (x^1 + y^1)$. The isomorphism on the third line exchanges $I^{n^0-1}$ with $I$, contributing to $\dim I^{n^0-1} \cdot\dim I$; and $I^{n^1 -1}$ with $J^0$, contributing to $\dim I^{n^1 -1} \cdot\dim J^0$.
\end{proof}

\section{The target weak Kan complexes}
\label{sec:three_wKan}

In this section we construct the target $\infty$-categories of Theorem~\ref{th:LieMon_to_fBialg} and Corollaries~\ref{cor:Manact_to_uBimodact} and \ref{cor:fun_Man}. We first introduce graded linear categories 
$\Rudmod$, $\Rudubimod$, $\Ruddbimod$, $\Rumod$ and $\Rdmod$. Taking their associated $dg$ categories of chain complexes and applying the $dg$ nerve construction will give us weak Kan complexes $\PreBialg$, $\uBimodact$, $\dBimodact$, $\PreAlg$ and $\PreCoalg$. We will then define our target categories as subcategories of those. We start by constructing $\fBialg$ in Section~\ref{ssec:fBialg}, and then explain how to obtain $\fAlg$ and $\fCoalg$ in Section~\ref{ssec:fAlg_fCoalg}.

\subsection{The weak Kan complex $\fBialg$}
\label{ssec:fBialg}

Geometrically speaking, for proving Theorem~\ref{th:LieMon_to_fBialg}, one follows the proof of Corollary~\ref{cor:fun_Man_Kom}, and replaces the line by biforests. Algebraically, this corresponds to replacing the category $R\text{-}mod$ by the following:

\begin{propdef}\label{propdef:udRmod}
The following defines a graded linear category, called the \emph{category of ascending-descending foresty $R$-modules}, and denoted $\Rudmod$:

\begin{itemize}
\item Objects consist of graded $R$-modules,

\item Given two such objects $A$ and $B$, let
\e
\hom_{\Rudmod}(A,B) :=\prod_{\kund, \lund} \hom_{R\text{-}mod}\left( (A^{b})^{\abs{\kund}} , (B^{\abs{\lund}})^{a}\right)[- (v(\kund) + v(\lund))] .
\e
We will denote its elements as $\varphi = \left\lbrace \varphi^{\kund}_{\lund} \right\rbrace_{\kund, \lund}$, and think of $\varphi^{\kund}_{\lund} $ as lying in a rectangular box as in \cite[Fig.~10]{biass}. Given a homogeneous element $\varphi^{\kund}_{\lund} $, $\deg \varphi^{\kund}_{\lund} $ will stand for its degree as an element of $\hom_{\Rudmod}(A,B)$, while we will denote $\abs{\varphi^{\kund}_{\lund} }$ its degree as an element of $\hom_{R\text{-}mod}\left( (A^{b})^{\abs{\kund}} , (B^{\abs{\lund}})^{a}\right)$. That is, $\deg \varphi^{\kund}_{\lund} = \abs{\varphi^{\kund}_{\lund} } - (v(\kund) + v(\lund)) $.

\item Given $\varphi = \left\lbrace \varphi^{\kund}_{\lund} \right\rbrace \colon A\to B$ and $\psi = \left\lbrace \psi^{\kund}_{\lund} \right\rbrace \colon B\to C$, define $\psi\circ \varphi \colon A\to C$ by:
\e
(\psi\circ \varphi)^{\kund}_{\lund} = \sum_{ \begin{subarray}{c} \kund = \kund^1 \sharp \kund^0 \\ \lund = \lund^0 \sharp \lund^1 \end{subarray}} (-1)^{s^{0,1}} \cdot  \psi^{\kund^0}_{\lund^0}\circ {\varphi}^{\kund^1}_{\lund^1},
\e
where, using notations as in Proposition~\ref{prop:or_glueing_maps},

\e\label{eq:sign_s}
s^{0,1} = \heartsuit_k^{0,1} + \heartsuit_l^{0,1} + y^0 (x^1 + y^1) + (x^0 + y^0) \cdot \deg (\varphi^{\kund^1}_{\lund^1}).
\e

\item The identity morphism of $A$ is:
\e
(id_A)^{\kund}_{\lund} = \begin{cases} id_{(A^b)^a}  &\text{if }\kund = \One_a \text{ and }\lund = \One_b ,\\
   0      &\text{otherwise. } 
 \end{cases}
\e
\end{itemize}

\end{propdef}
\begin{proof}
Let us prove associativity of composition, after some notational preparation. Assume we are given eight multi-indices such that $\kund = \kund^2 \sharp \kund^1 \sharp \kund^0$ and $\lund = \lund^0 \sharp \lund^1\sharp \lund^2 $. Denote, when it makes sense, $\kund^{ij} = \kund^j \sharp \kund^i$, $\lund^{ij} = \lund^i \sharp \lund^j$. Likewise, let $x^{ij}=v(\kund^{ij})= x^i + x^j$ and  $y^{ij}=v(\lund^{ij}) = y^i + y^j$. 
Let $s^{01,2}$ (resp. $s^{0,12}$) be the sign associated to 
$(\kund^{01},\lund^{01} )$,  $(\kund^2,\lund^2)$ and the relevant maps (resp. $ (\kund^0,\lund^0)$ and $ (\kund^{12},\lund^{12} )$). Likewise we denote $\heartsuit_k^{01,2}$, $\heartsuit_k^{0,12}$, $\heartsuit_l^{01,2}$, $\heartsuit_l^{0,12}$.

Say that a quantity $a^{0,1}$ is associative if $a^{01,2} + a^{0,1}=a^{0,12}+a^{1,2}$. We would like to show that $s^{0,1}$ is associative. 
Recall from \cite{biass} that $\heartsuit_k^{0,1}$ is defined as the signature of a permutation of vertices $\Vert (\kund^{01}) \to \Vert (\kund^0)\cup\Vert (\kund^1)$. 
Observe first that $\heartsuit_k^{0,1}$ is associative, since both $\heartsuit_k^{01,2}+ \heartsuit_k^{0,1}$ and $\heartsuit_k^{0,12}+ \heartsuit_k^{1,2}$ correspond to the same permutation, 
corresponding to the map 
\e
\Vert (\kund) \to \Vert (\kund^0)\cup\Vert (\kund^1)\cup\Vert (\kund^2).
\e
Likewise, $\heartsuit_l^{0,1}$ is associative. Observe now that if $b$ and $c$ are additive quantities, in the sense that $b^{01} = b^0 + b^1$ and $c^{01} = c^0 + c^1$, then $a^{0,1}= b^0 \cdot c^1$ is associative. Notice now that $s$ is a sum of quantities of the above forms, therefore it is associative. Associativity of composition follows:
\ea
((\varphi\circ \chi )\circ \psi)^{\kund}_{\lund}&= \sum_{ \begin{subarray}{c} \kund = \kund^2 \sharp \kund^{01} \\ \lund = \lund^{01} \sharp \lund^2 \end{subarray}} (-1)^{s^{01,2}} \cdot  (\varphi\circ \chi )^{\kund^{01}}_{\lund^{01}}\circ {\psi}^{\kund^2}_{\lund^2} \nonumber\\
 &= \sum_{ \begin{subarray}{c} \kund = \kund^2 \sharp \kund^{1}\sharp \kund^{0} \\ \lund = \lund^{0} \sharp  \lund^{1} \sharp \lund^2 \end{subarray}} (-1)^{s^{01,2} +s^{0,1}} \cdot  \varphi^{\kund^0}_{\lund^0} \circ \chi^{\kund^1}_{\lund^1} \circ {\psi}^{\kund^2}_{\lund^2} \nonumber\\
 &= \sum_{ \begin{subarray}{c} \kund = \kund^{12} \sharp \kund^{1} \\ \lund = \lund^{0} \sharp \lund^{12} \end{subarray}} (-1)^{s^{0,12}} \cdot  \varphi^{\kund^{0}}_{\lund^{0}}\circ (\chi \circ {\psi})^{\kund^{12}}_{\lund^{12}} \nonumber\\
 &= (\varphi\circ (\chi \circ \psi))^{\kund}_{\lund} .\nonumber
\ea
The fact that $id_A$ are identities is straightforward to check.
\end{proof}

Therefore, it has an associated $dg$-category of chain complexes $\Ch_{\Rudmod}$, on which the $dg$ nerve construction applies:

\begin{defi}\label{def:PreBialg}
Let the weak Kan complex $\PreBialg = N_{dg} (\Ch_{\Rudmod})$.
\end{defi}
The above could serve as a target for Theorem~\ref{th:LieMon_to_fBialg}, however it misses the important relationship with \Ainf -(co)algebras and bialgebras that the Morse complexes have (Propositions~3.5 and 3.7 in \cite{biass}). Therefore we will consider a subcategory of it. Let us first unravel the definition.

The objects of $\Ch_{\Rudmod}$ are pairs $(A, \alpha)$, with $\alpha \colon A \to A$ of degree $-1$, i.e. $\abs{ \alpha^{\kund}_{\lund} } = v(\kund) + v(\lund) -1 = \dim J(0)^{\kund}_{\lund} $; and such that $\alpha \circ \alpha = 0$, i.e.
\e\label{eq:alpha_circ_alpha=0}
0= \sum (-1)^{s^{0,1}} \alpha^{0} \circ \alpha^{1},
\e
where we sum over all splittings of $\kund$ and $\lund$, and $\alpha^{i} = \alpha^{\kund^{i}}_{\lund^{i}} $. In this case,
\ea
s^{0,1} &= \heartsuit_k + \heartsuit_l + y^0 (x^1 + y^1) + (x^0 + y^0)  \\
&= 1+ \rho^{0,1},\nonumber
\ea
with $\rho^{0,1}$ the orientation of the gluing map  (\ref{eq:rho}), for $n_0=n_1=0$. Therefore (up to an overall minus sign), (\ref{eq:alpha_circ_alpha=0}) is the coherence relation of $f$-bialgebras as in \cite[Def.~3.2]{biass}.

If $(A, \alpha)$ and $(B, \beta)$ are as above, the space of morphisms $\varphi\colon (A, \alpha) \to (B, \beta) $ consists of families of maps $\varphi^{\kund}_{\lund} \colon  (A^{b})^{\abs{\kund}} \to (B^{\abs{\lund}})^{a}$, and is equipped with the differential $\partial(\varphi) = \beta \circ\varphi  - (-1))^{\deg \varphi} \varphi \circ \alpha  $, i.e.
\e\label{eq:partial_varphi}
\partial(\varphi)^{\kund}_{\lund} = \sum (-1)^{s^{0,1}_{\beta}} \beta^{0} \circ\varphi^1  + (-1))^{1+ \deg \varphi + s^{0,1}_{\alpha}} \varphi^{0} \circ \alpha^1 , 
\e
with the same conventions as before, and
\ea
s^{0,1}_{\beta}  &=  \heartsuit_k + \heartsuit_l + y^0 (x^1 + y^1) + (x^0 + y^0) \cdot \deg (\varphi), \\
 s^{0,1}_{\alpha} &=  \heartsuit_k + \heartsuit_l + y^0 (x^1 + y^1) + (x^0 + y^0).
\ea
In particular, if $\deg \varphi = 0$, i.e. $\abs{\varphi^{\kund}_{\lund} } = v(\kund) + v(\lund)$; (\ref{eq:partial_varphi})  is the coherence relation of morphisms of  $f$-bialgebras as in \cite[Def.~3.9]{biass}.

Now, an $n$-simplex in the $dg$-nerve consists in chain complexes $(A_{0}, \alpha_{0})$, ..., $(A_{n}, \alpha_{n})$ together with, for each $\sigma = [\sigma_0, \ldots , \sigma_m]\subset [0, \ldots , n]$ of $\dim \sigma \geq 1$, morphisms $\varphi_{\sigma} \colon A_{\sigma_m} \to A_{\sigma_0}$ of degree $m-1$, i.e. 

\e
\abs{(\varphi_{\sigma})^{\kund}_{\lund} } = v(\kund) + v(\lund) +m -1 = \dim J(m)^{\kund}_{\lund},
\e
and such that 
\e\label{eq:cohrel_varphi_sigma}
\partial (\varphi_{\sigma}) = \sum_{i=1}^{m-1} (-1)^i (\varphi_{\partial_i \sigma} - \varphi_{\sigma_{\leq i}} \circ \varphi_{\sigma_{\geq i}}).
\e
If $\sigma = [\sigma_0]$, letting $\varphi_{\sigma} = \alpha_{\sigma_0}$, moving $\partial (\varphi_{\sigma})$ to the right and absorbing it in the second sum (corresponding to $\dim \sigma^0 = 0$ or  $\dim \sigma^1 = 0$); the equation becomes, for all $\ddhat = (\sigma,\kund, \lund)$:
\e\label{eq:cohrel_R_ddhat}\tag{$R_{\ddhat}$}
0=\sum_{i=1}^{m-1} (-1)^i \varphi_{\partial_i \ddhat} +\sum_{\ddhat = \ddhat^1 \sharp \ddhat^0} (-1)^{\rho^{0,1}}   \varphi_{\ddhat^0} \circ \varphi_{\ddhat^1}  ,
\e
where  $\partial_i \ddhat = (\partial_i \sigma, \kund, \lund)$, and $\ddhat^f = (\sigma^f, \kund^f, \lund^f)$.

\begin{propdef}\label{propdef:fBialg}

Let $\fBialg$ be the simplicial subset of  $\PreBialg$ such that its $n$-simplices are $\left\lbrace A_i, \varphi_{\sigma} \right\rbrace$ as above, satisfying furthermore that:
\begin{itemize}
\item If $\dim \sigma = 0$, $\varphi_{\sigma}$ satisfies the simplification relations of $f$-bialgebras, which we recall below.
\item If $\dim \sigma = 1$, $\varphi_{\sigma}$ satisfies the simplification relations of morphisms of $f$-bialgebras, which we recall below.
\end{itemize}
This also defines a weak Kan complex.

The simplification relations  for $f$-bialgebras are \cite[Def.~3.2]{biass}:
\begin{itemize}
\item For $a,b\geq 1$, $\alpha^{\One_a}_{\One_b}$ coincides with the tensor product differential induced by the differential $\partial_A = \alpha^1_1$ of $A$. We will refer to this condition as $(T^a_b)$.
\item If $\kund = \One_a$ and $\lund = (1, \ldots 1, l_j , 1, \ldots, 1)$, with $l_j\geq 2$, then
\e\label{eq:V_a_lund}\tag{$V^{a}_{\lund}$}
{\alpha}^{\kund}_{\lund} \simeq id_{(A^{j-1})^a} \otimes {\alpha}^{\kund}_{l_j} \otimes id_{(A^{b-j})^a},
\e
where $\simeq$ is in the sense of (\ref{eq:iso_split_b}).
\item Likewise, if  $\lund = \One_b$ and $\kund = (1, \ldots 1, k_i , 1, \ldots, 1)$, with $k_i\geq 2$, then
\e\label{eq:V_kund_b}\tag{$V^{\kund}_{b}$}
{\alpha}^{\kund}_{\lund} \simeq id_{(A^{b})^{i-1}} \otimes {\alpha}^{k_i}_{\lund} \otimes id_{(A^{b})^{a-i}},
\e
where $\simeq$ is in the sense of (\ref{eq:iso_split_a}).
\item If either $\kund = \One_a$ and $\lund$ has at least two entries $\geq 2$, or vice versa, then ${\alpha}^{\kund}_{\lund} =0$. We will refer to these conditions as $(V^{a}_{\lund})$ and $(V^{\kund}_{b})$ as well.

\item (vertical tree deletion) Suppose that $\lund$ is an almost vertical forest, i.e. only with 1 and 2. Let $\kund$ have a vertical tree at position $i$, i.e. $k_i = 1$, and let $\widehat{\kund} = (k_1, \ldots , \widehat{k_i} , \ldots, k_a)$, then
\e\label{eq:D_kund_i_lund}\tag{$D^{\kund,i}_{\lund}$}
{\alpha}^{\kund}_{\lund} \simeq {\alpha}^{\widehat{\kund}}_{\lund} \otimes  {\widetilde{\alpha}}^{1}_{\lund},
\e
where ${\widetilde{\alpha}}^{1}_{\lund} =  {\widetilde{\alpha}}^{1}_{l_1} \otimes \cdots \otimes {\widetilde{\alpha}}^{1}_{l_b}$, with ${\widetilde{\alpha}}^{1}_{1}= id_A$ and ${\widetilde{\alpha}}^{1}_{2}= \alpha^1_2$.

Likewise, if $\kund$ is an almost vertical forest. Let $\lund$ have a vertical tree at position $i$, i.e. $l_i = 1$, and let $\widehat{\lund} = (l_1, \ldots , \widehat{l_i} , \ldots, l_b)$, then
\e\label{eq:D_kund_lund_i}\tag{$D^{\kund}_{\lund,i}$}
{\alpha}^{\kund}_{\lund} \simeq {\alpha}^{\kund}_{\widehat{\lund}} \otimes  \widetilde{\alpha}^{\kund}_{1},
\e
where ${\widetilde{\alpha}}^{\kund}_{1} =  {\widetilde{\alpha}}^{k_1}_{1} \otimes \cdots \otimes {\widetilde{\alpha}}^{k_a}_{}$, with ${\widetilde{\alpha}}^{1}_{1}= id_A$ and ${\widetilde{\alpha}}^{2}_{1}= \alpha^2_1$.
\end{itemize}

The simplification relations for morphisms of $f$-bialgebras are \cite[Def.~3.9]{biass}:

\begin{itemize}
\item If $\lund$ is vertical, then
\e\label{eq:W_kund_b}\tag{$W^{\kund}_{b}$}
\varphi^{\kund}_{\lund}  \simeq 
\varphi^{k_1}_{\lund}  \otimes \cdots \otimes  \varphi^{k_a}_{\lund}.
\e

\item If $\kund$ is vertical, then
\e\label{eq:W_a_lund}\tag{$W^{a}_{\lund}$}
\varphi^{\kund}_{\lund}  \simeq  \varphi^{\kund}_{l_1} \otimes \cdots \otimes \varphi^{\kund}_{l_b} .
\e
\end{itemize}

\end{propdef}

\begin{proof}
Since the horn-filling property is satisfied in $\PreBialg$, and since the extra assumptions are only on objects and 1-morphisms; 
 it is enough to show that if $\varphi$ and $\psi$ are composable 1-morphisms satisfying the simplification relations, then so does $\varphi\circ\psi$.

Assume $\lund = \One_b$, observe that in this case all signs $s^{0,1}$ are trivial. Then,
\ea
(\varphi\circ\psi)^{\kund}_{\One_b} &= \sum_{\kund = \kund^1 \sharp \kund^0}  \varphi^{\kund^0}_{\One_b} \circ \psi^{\kund^1}_{\One_b}  \nonumber\\  &\simeq \sum_{\kund = \kund^1 \sharp \kund^0} ( \varphi^{k^0_1}_{\One_b} \otimes \cdots \otimes \varphi^{k^0_{a^0}}_{\One_b}  ) \circ ( \psi^{k^1_1}_{\One_b} \otimes \cdots \otimes \psi^{k^1_{a^1}}_{\One_b}  )\nonumber\\ 
&= \left( \sum_{(k_1) = \kund^1_1 \sharp (k^0_1)}  \varphi^{k^0_1}_{\One_b} \circ \psi^{\kund^1_1}_{\One_b}  \right) \otimes \cdots \otimes  \left( \sum_{(k_{a^0}) = \kund^1_{a^0} \sharp (k^0_{a^0})}  \varphi^{k^0_{a^0}}_{\One_b} \circ \psi^{\kund^1_{a^0}}_{\One_b} \right)   \nonumber\\
&=  (\varphi\circ\psi)^{k_1}_{\One_b} \otimes \cdots \otimes  (\varphi\circ\psi)^{k_a}_{\One_b} , \nonumber
\ea
where from the second to the third line we wrote $\kund^0 = (k^0_1, \ldots , k^0_{a^0})$, and decomposed $\kund^1$ to  $\kund^1_1$, ..., $\kund^1_{a^0}$, where $\kund^1_{i}$ corresponds to the sub-forest growing on top of the $i$-th tree of $\kund^0$.

The proof of the simplification relation when $\kund = \One_a$ is similar.
\end{proof}

\subsection{The weak Kan complexes $\uBimodact$ and $\dBimodact$}
\label{ssec:uBimodact_dBimodact}

We now define $\Rudubimod$ and $\uBimodact$, which are $u$-bimodule counterparts of $\Rudmod$ and $\fBialg$. One can also define $\Ruddbimod$ and $\dBimodact$ analogously by exchanging the roles of $\kund$ and $\lund$. As this would be for Morse cochain complexes, we leave the details to the contravariant readers.

We denote bimodule multi-indices $\kundb = (\kundl |\epsilon | \kundr)$. 
Let also $\dd = (\kundb,\lund)$, $v(\dd) = v(\kund) + v(\lund)$.

\begin{propdef}\label{propdef:Rudubimod}

Let $\Rudubimod$ be the graded linear category whose:
\begin{itemize}
\item Objects consist of triples $(A,M,B)$ of graded $R$-modules,

\item Given two triples $T=(A,M,B)$ and  $U=(C,N,D)$, let 
\ea
&\hom_{\Rudubimod}(T,U ) \\
&:=\prod_{\dd = (\kundb, \lund)} \hom_{R\text{-}mod}\left( In(\dd) , Out(\dd) \right)[- v(\dd)]\text{, with} \\
In&(\dd) = (A^{b})^{\abs{\kundl}} \otimes (M^{b})^{\epsilon} \otimes (B^{b})^{\abs{\kundr}} , \\
Out&(\dd) = (C^{\abs{\lund}})^{a^l}\otimes (N^{\abs{\lund}})^{\epsilon}\otimes (D^{\abs{\lund}})^{a^r}
\ea

\item Given $\varphi = \left\lbrace \varphi^{\kundb}_{\lund} \right\rbrace \colon T\to U$ and $\psi = \left\lbrace \psi^{\kundb}_{\lund} \right\rbrace \colon U\to V$, define $\psi\circ \varphi \colon T\to V$ by:
\e
(\psi\circ \varphi)_{\dd} = \sum_{ \dd = \dd^1 \sharp \dd^0} (-1)^{s^{0,1}} \cdot  \psi_{\dd^0}\circ {\varphi}_{\dd^1},
\e
where $s^{0,1}$ is given by (\ref{eq:sign_s}).

\item The identity morphism of $T$ is:
\e
(id_T)_{\dd} = \begin{cases} id_{In(\dd)}  &\text{if }\kund = \One_a \text{ and }\lund = \One_b ,\\
   0      &\text{otherwise. } 
 \end{cases}
\e
\end{itemize}

\end{propdef}
\begin{proof}Same as Proposition-Definition~\ref{propdef:udRmod}.
\end{proof}

\begin{defi}\label{def:uBimodact} Let $\PreuBimodact = N_{dg}(Ch(\Rudubimod))$. Let $\uBimodact$ be the simplicial subset of it, with $n$-simplices $\left\lbrace T_i, \mu_\sigma \right\rbrace$ such that
\begin{itemize}
\item if $\dim \sigma = 0$, $\mu_\sigma$ satisfies the simplification relations of $u$-bimodules given in \cite[Def.~3.13]{biass},
\item if $\dim \sigma = 1$, $\mu_\sigma$ satisfies the simplification relations of morphisms of $u$-bimodules given in \cite[Def.~3.15]{biass}.
\end{itemize}
\end{defi}

\begin{remark} The assignment $(A,M,B) \mapsto A\oplus M \oplus B$ gives a linear functor $\Rudubimod \to \Rudmod$, which induces a simplicial map $\uBimodact \to \fBialg$.

\end{remark}

\subsection{The weak Kan complexes $\fAlg$ and $\fCoalg$}
\label{ssec:fAlg_fCoalg}

Consider the following two non-full non-unital\footnote{in the sense that the identity morphisms are different} subcategories of $\Rudmod$. Loosely speaking, $\Rumod$ (resp. $\Rdmod$) is obtained by setting $\lund = (1)$ (resp. $\kund = (1)$).

\begin{defi}
Let the category of \emph{ascending (foresty) $R$-modules} $\Rumod$ be such that:
\begin{itemize}
\item Objects consist of graded $R$-modules,

\item Given two such objects $A$ and $B$, let
\e
\hom_{\Rumod}(A,B) :=\prod_{\kund} \hom_{R\text{-}mod}\left( A^{\abs{\kund}} , B^{a}\right)[- v(\kund) ] .
\e

\item Given $\varphi = \left\lbrace \varphi^{\kund} \right\rbrace \colon A\to B$ and $\psi = \left\lbrace \psi^{\kund} \right\rbrace \colon B\to C$, define $\psi\circ \varphi \colon A\to C$ by:
\e
(\psi\circ \varphi)^{\kund} = \sum_{ \begin{subarray}{c} \kund = \kund^1 \sharp \kund^0  \end{subarray}} (-1)^{s^{0,1}} \cdot  \psi^{\kund^0}\circ {\varphi}^{\kund^1},
\e
where, using notations as in Proposition~\ref{prop:or_glueing_maps},

\e
s^{0,1} = \heartsuit_k^{0,1}  + x^0  \cdot \deg (\varphi^{\kund^1}).
\e

\item The identity morphism of $A$ is:
\e
(id_A)^{\kund} = \begin{cases} id_{A^a}  &\text{if }\kund = \One_a  ,\\
   0      &\text{otherwise. } 
 \end{cases}
\e
\end{itemize}

Likewise, define the category of \emph{descending (foresty) $R$-modules} $\Rdmod$, with same objects and such that:
\ea
\hom_{\Rdmod}(A,B) &:=\prod_{ \lund} \hom_{R\text{-}mod}\left( A^{b} , B^{\abs{\lund}})\right)[-  v(\lund)] , \\
(\psi\circ \varphi)_{\lund} &= \sum_{ \begin{subarray}{c} \lund = \lund^0 \sharp \lund^1 \end{subarray}} (-1)^{s^{0,1}} \cdot  \psi_{\lund^0}\circ {\varphi}_{\lund^1}\text{, with}\\
s^{0,1} &= \heartsuit_l^{0,1} + y^0 \cdot (y^1 + \deg (\varphi_{\lund^1}) ) , \\
(id_A)_{\lund} &= \begin{cases} id_{A^b}  &\text{if }\lund = \One_b ,\\
   0      &\text{otherwise. } 
 \end{cases}
\ea
\end{defi}

\begin{defi} Let $\PreAlg = N_{dg}(\Ch_{\Rumod})$, and let $\fAlg$ be the simplicial subset such that:
\begin{itemize}
\item Objects $(A, \alpha)$ satisfy:
\e
\alpha^{\kund} = \begin{cases} \sum_{i=1}^{a} (id)^{i-1} \otimes \alpha^1 \otimes  (id)^{i-1} &\text{if }\kund = \One_a, \\
(id)^{i-1} \otimes \alpha^{k_i} \otimes  (id)^{i-1} &\text{if }\kund = (1, \ldots 1, k_i ,  1, \ldots 1 ), \\
0 &\text{if }\tilde{n}(\kund) \geq 2 .
\end{cases}
\e
That is, the $\alpha^{\kund}$ are completely determined by the $\alpha^{k}$, which form an \Ainf -algebra, with the sign convention of \cite[Prop.~3.5]{biass}.
\item 1-morphisms $\varphi\colon A\to B$ satisfy $\varphi^{\kund} = \varphi^{k_1} \otimes \cdots \otimes \varphi^{k_a} $. That is, the $\alpha^{\kund}$ are completely determined by the $\varphi^{k}$, which form an \Ainf -morphism, with the sign convention of \cite[Prop.~3.11]{biass}.
\end{itemize}

Likewise, let $\PreCoalg = N_{dg}(\Ch_{\Rdmod})$, and let $\fCoalg$ be the simplicial subset whose objects and 1-morphisms satisfy similar relations.
\end{defi}

One has forgetful functors, and their induced forgetful $\infty$-functors:
\e
\xymatrix{
\Rudmod \ar[r]\ar[d]  & \Rumod \ar[d] & \fBialg \ar[r]\ar[d]  & \fAlg \ar[d]\\
\Rdmod \ar[r] & \Rmod & \fCoalg \ar[r] & N_{dg}(\Ch_{\Rmod}) .
}
\e

\section{Tautological families of graphs, perturbations}
\label{sec:tautol_fam_pert}

\subsection{Tautological families of graphs}
\label{ssec:Tautological_families_of_graphs}
We now construct the source space for our perturbations, called the ``tautological grafted biforest''. It is a space $\Gcal$ endowed with a projection to $J = \coprod_{n, \kund, \lund} J(n)^{\kund}_{\lund}$, such that the fiber over a point is a grafted graph corresponding to that point. We construct it step by step.

First, fix $\kund$, and let $\varphi$ be an isomorphism class of trivalent forests of type $\kund$. Let $U^{\kund}_{\varphi} \subset U^{\kund}$ stand for the corresponding chamber of ascending forests of the form $(\varphi,h)$, with some height function $h\colon \Vert (\varphi) \to \rr$. 

Let $e\in \Edges(\varphi)$, recall that it has a source  $s(e) \in \Vert(\varphi) \cup \Leaves(\varphi) $ and a target $t(e) \in \Vert(\varphi) \cup \Roots(\varphi) $. Let the corresponding heights 
\ea
h_s(e) = \begin{cases} h(s(e)) \\ +\infty \end{cases}\text{and }\ \  & h_t(e) = \begin{cases} h(t(e)) \\ -\infty \end{cases} ,
\ea
where these quantities are equal to $\pm \infty$ when the corresponding endpoint is not a vertex. 
Define
\e
\Ucal^{\kund}_{\varphi, e} := \left\lbrace (h,z) \ : \  (\varphi,h)\in U^{\kund}_{\varphi} \text{ and } h_t(e) \leq z \leq h_s(e)  \right\rbrace  ,
\e
so that $\Ucal^{\kund}_{\varphi, e}$ projects to $U^{\kund}_{\varphi}$, and the fiber over $(\varphi,h)$ is the interval $[ h_t(e) , h_s(e)  ]$ (excluding the bound when infinite). Let then 
\e
\Ucal^{\kund}_{\varphi} := \coprod_e \Ucal^{\kund}_{\varphi, e} ,
\e
and let the \emph{tautological ascending forest}:
\e
\Ucal^{\kund} := \left(  \coprod_{\varphi} \Ucal^{\kund}_{\varphi} \right) /\sim ,
\e
where if $(\varphi,h) = (\varphi',h')\in U^{\kund}$, and an edge $e$ of $\varphi$ corresponds to an edge $e'$ of $\varphi'$ we identify the corresponding intervals.

The space $\Ucal^{\kund}$ projects to $U^{\kund}$, has a height function $h^{\kund} \colon\Ucal^{\kund} \to \rr$ corresponding to projecting to the $z$ coordinate, and the fiber over $(\varphi,h)$ is a union of intervals $[ h_t(e) , h_s(e)  ]$ for each edge $e$. It corresponds to the ascending graph, disconnected at its vertices.

Likewise, one has a \emph{tautological descending forest} $(\Dcal_{\lund}, h_{\lund}) = (\Ucal^{\lund}, -h^{\lund})$, with the same projection to $D_{\lund} \simeq U^{\lund}$. Its fibers correspond to the (disconnected) descending graph.

Consider now the fiber product with respect to height functions: 
\e
\widetilde{\Gcal}^{\kund}_{\lund} = \Ucal^{\kund} \times_{\rr} \Dcal_{\lund},
\e
which inherits a height function denoted $\widetilde{h}^{\kund}_{\lund}$, and a projection to $J^{\kund}_{\lund} = U^{\kund} \times D_{\lund} $. Observe that the fibers correspond to Henriques intersections of the corresponding ascending and descending forests.

We now want to form ${\Gcal(n)}^{\kund}_{\lund}$. Consider first the case $n=0$, and assume that $c(\kund, \lund) = 1$ (otherwise let ${\Gcal(n)}^{\kund}_{\lund} = \emptyset$). Observe that the $\rr$ action on $J^{\kund}_{\lund}$ lifts to $\widetilde{\Gcal}^{\kund}_{\lund}$, Let ${\Gcal(0)}^{\kund}_{\lund}$ be the quotient
\e
{\Gcal(0)}^{\kund}_{\lund} = \widetilde{\Gcal}^{\kund}_{\lund} /\rr .
\e

Assume now that $n\geq 1$. Consider first
\e
\widetilde{\Gcal}(n)^{\kund}_{\lund} = I^{n-1} \times \widetilde{\Gcal}^{\kund}_{\lund} .
\e
Let  ${\Hcal(n)}^{\kund}_{\lund} \subset \widetilde{\Gcal}(n)^{\kund}_{\lund}$ correspond to the hypersurface of grafting heights, defined as follows. If $\Lund \in I^{n-1}$, let 
\e
H(\Lund) := \left\lbrace h_{01},  h_{12}, \ldots , h_{(n-1)n} \right\rbrace \subset \rr
\e
stand for the corresponding set of grafting heights, as defined in (\ref{eq:grafting_heights}). Let then 
\e
{\Hcal(n)}^{\kund}_{\lund} := \left\lbrace (\Lund, g) \ : \ \widetilde{h}^{\kund}_{\lund}(g) \in H(\Lund) \right\rbrace.
\e
Define then 
\e
{\Gcal}(n)^{\kund}_{\lund} =  \widetilde{\Gcal}(n)^{\kund}_{\lund} \setminus\!\!\!\setminus {\Hcal(n)}^{\kund}_{\lund},
\e
where the symbol $\setminus\!\!\setminus$ stands for ``cutting'', i.e. removing and gluing back twice.

To summarize, we have constructed a space ${\Gcal}(n)^{\kund}_{\lund}$ with a height function 
\e
{h}(n)^{\kund}_{\lund} \colon {\Gcal}(n)^{\kund}_{\lund} \to \rr,
\e
and a projection ${\Gcal}(n)^{\kund}_{\lund} \to {J}(n)^{\kund}_{\lund}$ whose fibers correspond to the Henriques intersection graphs, disconnected at vertices and grafting levels.

Furthermore, cutting along $ {\Hcal(n)}^{\kund}_{\lund}$ divides  ${\Gcal}(n)^{\kund}_{\lund}$ into levels ${\Gcal}(n;i)^{\kund}_{\lund}$, for $i = 0, \ldots, n$, which are mapped to $[ h_{(i-1)i}, h_{i(i+1)} ]$ by ${h}(n)^{\kund}_{\lund}$.

\subsection{Perturbations}
\label{ssec:Perturbations}
Let us now define the perturbation spaces that we will use in order to prove Theorem~\ref{th:LieMon_to_fBialg}.

Let us write
\e
J(n) = \coprod_{ \kund, \lund} J(n)^{\kund}_{\lund} \ \ \ ;\ \ \ \Gcal(n) = \coprod_{ \kund, \lund} \Gcal(n)^{\kund}_{\lund} ,
\e
and let the unions over all sub-simplices 
\e
J([n]) = \coprod_{ \sigma = [\sigma_0, \ldots , \sigma_m] \subset [n] } J_{\sigma} (m) \ \ \ ;\ \ \ \Gcal([n]) = \coprod_{ \sigma = [\sigma_0, \ldots , \sigma_m] \subset [n] } \Gcal_{\sigma}(m),
\e
where $J_{\sigma} (m) \simeq J (m)$ and $\Gcal_{\sigma}(m) \simeq \Gcal(m)$.

\begin{defi} Let $\LieMon'$ stand for the category whose objects are triples $(G,f,V)$ of a compact Lie monoid $G$, endowed with a Morse function $f$ and a Palais-Smale pseudo-gradient $V$; and whose morphisms are smooth morphisms as in $\LieMon$.
\end{defi}

Let $S\in N(\LieMon')_n$ be an $n$-simplex. That is $S = \lbrace G_i,f_i,V_i, \phi_{ij} \rbrace$, with $i=0, \ldots , n-1$; $\phi_{i(i+1)} \colon G_{i+1} \to G_i$, and for $i<j$,  
\e
\phi_{ij} = \phi_{i(i+1)} \circ \phi_{(i+1)(i+2)} \circ \cdots \circ \phi_{(j-1)j} \colon G_{j} \to G_i .
\e

Our perturbation space involved in the moduli spaces  in the next section is the subspace
\e
\Xfrak (S) \subset C^{\infty} ( \Gcal([n]) , \Xfrak (G_0) \sqcup \cdots \sqcup  \Xfrak (G_n) )
\e
of domain-dependent vector fields $V$ safisfying several conditions that we describe below.

\begin{remark}
Observe that $\Gcal([n])$ is a disjoint union of manifolds with corners, we consider smooth maps with respect to these structures.
\end{remark}

\noindent\textbf{Level condition:} $V$ maps a level $\Gcal_{\sigma}(m;i)$ of $\Gcal_{\sigma}(m)$ 
to the corresponding space $\Xfrak (G_{\sigma_i})$.

The next condition is to guarantee that breaking occurs as one would expect. Let  $\Gcal^{\rm br}([n]) \subset \Gcal([n])$ (resp. $\Gcal^{\rm pert}([n]) \subset \Gcal([n])$) corresponds to points which are at distance $\geq 1$ (resp. $\leq 1$) of the endpoints of the interval they belong, in their corresponding fiber.

\noindent\textbf{Breaking condition:} $V$ maps $\Gcal^{\rm br}_\sigma(m;i)$ to the given pseudo-gradient $V_{\sigma_i}$. That is, we will only perturb on $\Gcal^{\rm pert}([n])$, as in \cite{abouzaid2011topological,Mazuir1}.

The next condition is for when the height of a level shrinks to zero. Recall that if $m\geq 2$, $1\leq i \leq m-1$, as unoriented manifolds,
\e
\partial_i J_{\sigma}(m)^{\kund}_{\lund} \simeq J_{\partial_i \sigma}(m-1)^{\kund}_{\lund}.
\e

Under this identification, we have that 
\e
{\Gcal_{\sigma}(m)^{\kund}_{\lund}}_{| \partial_i J(m)  } \simeq \Gcal_{\partial_i \sigma}(m-1)^{\kund}_{\lund} \sqcup {\Gcal_{\sigma}(m;i)^{\kund}_{\lund}}_{| \partial_i J(m)  },
\e
where the fibers of ${\Gcal_{\sigma}(m;i)^{\kund}_{\lund}}_{| \partial_i J(m)  }$ are a finite number of points.

\noindent\textbf{Shrinking condition:} $V$ commutes with the inclusions 
\e
\Gcal_{\partial_i \sigma}(m-1)^{\kund}_{\lund} \hookrightarrow {\Gcal_{\sigma}(m)^{\kund}_{\lund}}_{| \partial_i J(m)  } .
\e

The above conditions guarantee that we can recognize the terms $\partial (\varphi_\sigma )^{\kund}_{\lund}$ and $ (\varphi_{\partial_i \sigma} )^{\kund}_{\lund}$ in the coherence relations (\ref{eq:cohrel_varphi_sigma}). For the remaining summands, let us introduce the following terminology.
\begin{defi}\label{def:Lambda_spacing} Consider a grafted biforest $p = (\Lund , B) \in J(n )^{\kund}_{\lund} $ (possibly ungrafted, i.e. $n=0$ and $\Lund = \emptyset$). For $\Lambda >1$, say that $p$ is $\Lambda$-spaced at height $H\in \rr$ if 
\e
{\rm dist} ( \left\lbrace H\right\rbrace, H(\Lund) \cup h(\Vert B) )\geq \Lambda.
\e 
\end{defi}
In this case, the perturbation neighborhoods can be written as disjoint unions, with $\ddhat= (\sigma, \kund, \lund)$:
\e
\left(\Gcal_{\ddhat}^{\rm pert} \right)_p = \left(\Gcal_{\ddhat}^{\rm pert} \right)_p^0  \sqcup \left(\Gcal_{\ddhat}^{\rm pert} \right)_p^1 , 
\e
each part respectively corresponding to the part below and above $H$ (with respect to the height function).

Furthermore, the level $H$ induces a splitting $\ddhat= \ddhat^1 \sharp \ddhat^0$, by which we mean:
\begin{itemize}
\item $\sigma^0 = [\sigma_0, \ldots , {\sigma}_i ]$ and $ \sigma^1 = [{\sigma}_i , \ldots , \sigma_m]$, for some $0\leq i \leq m = \dim \sigma$,
\item $\kund =\kund^1 \sharp \kund^0$, and $\lund =\lund^0 \sharp \lund^1$.
\end{itemize}  
The grafted biforest $p$ splits accordingly to $p^{0}\in J_{\ddhat^0}$ and $p^{1}\in J_{\ddhat^1}$. 
One then has, \emph{in most cases}, identifications
\ea
\left(\Gcal_{\ddhat}^{\rm pert} \right)_p^0  &\simeq \left(\Gcal_{\ddhat^0}^{\rm pert} \right)_{p^0 }, \label{eq:Gcal_pert_0} \\
\left(\Gcal_{\ddhat}^{\rm pert} \right)_p^1  &\simeq  \left(\Gcal_{\ddhat^1}^{\rm pert} \right)_{p^1 } .\label{eq:Gcal_pert_1}
\ea
This is true as long as, for each  floor $f=0,1$, $\dim \sigma^f >0$ or $c(\kund^f, \lund^f) =1$. If not, 
recall that we have set $J_{\ddhat^f} = \emptyset$. In this case, if $m^f = 0$, $\lund^f = \One_{b^f}$ and $\kund^f = (k^f_1, \ldots , k^f_{a^f})$,
\e\label{eq:Gcal_pert_degen_l}
\left(\Gcal_{\ddhat}^{\rm pert} \right)_p^f \simeq \left(\Gcal_{\ddhat^f_1}^{\rm pert} \right)_{p^f_1 } \sqcup \cdots \sqcup \left(\Gcal_{\ddhat^f_{a^f}}^{\rm pert} \right)_{p^f_{a^f} },
\e
with $\ddhat^f_i = (\sigma^f, k^f_i, \One_{b^f})$ and $p^f_i \in J_{\ddhat^f_i}$. 
If $m^f = 0$, $\kund^f = \One_{a^f}$ and $\lund^f = (l^f_1, \ldots , l^f_{b^f})$,
\e\label{eq:Gcal_pert_degen_k}
\left(\Gcal_{\ddhat}^{\rm pert} \right)_p^f \simeq \left(\Gcal_{\ddhat^f_1}^{\rm pert} \right)_{p^f_1 } \sqcup \cdots \sqcup \left(\Gcal_{\ddhat^f_{b^f}}^{\rm pert} \right)_{p^f_{b^f} },
\e
with $\ddhat^f_j = (\sigma^f, \One_{a^f}, l^f_j )$ and $p^f_j \in J_{\ddhat^f_j}$.

\noindent\textbf{Spacing condition:} There exists $\Lambda_0 >1$ such that for any $\Lambda \geq \Lambda_0$ and any $\Lambda$-separated grafted biforest as above, $V$ commutes with the identifications (\ref{eq:Gcal_pert_0}), (\ref{eq:Gcal_pert_1}), (\ref{eq:Gcal_pert_degen_l}) and  (\ref{eq:Gcal_pert_degen_k}).

\begin{remark}One could have extended $\Gcal_{\dd} \to J_{\dd}$ to the partial compactification $\left( \overline{J}_{\dd} \right)_{\leq 1}$ constructed in Section~\ref{sec:higher_bimultipl}. The above condition would then ensure that $V$ extends over it. 
\end{remark}

\begin{remark}For any grafted biforest, one can always find a $\Lambda$-spacing level by taking $\abs{H}$ large enough. For such a level, the condition is vacuous. 
\end{remark}

The above conditions would be enough to guarantee that we construct an $\infty$-functor to $N_{dg}(\PreBialg)$. To ensure the extra conditions of $\fBialg$ are satisfied, we impose a few more conditions.

Equations (\ref{eq:V_kund_b}), (\ref{eq:V_a_lund}) are already guaranteed by the spacing condition. Let us turn to (\ref{eq:D_kund_i_lund}) and (\ref{eq:D_kund_lund_i} ).

Assume that $m=\dim \sigma = 0$, $\kund$ is such that $k_i=1$ for some $i$, and $\lund$ is almost vertical ($l_j=1$ or $2$). 
With $\widehat{\kund} = (k_1, \ldots , \widehat{k}_i, \ldots , k_a)$, one has:
\e
\Gcal_{\sigma}(0)^{\kund}_{\lund} \simeq \Gcal_{\sigma}(0)^{\widehat{\kund}}_{\lund}  \sqcup \left( J_{\sigma}(0)^{\kund}_{\lund} \times   \coprod_{j=1}^b \Gcal_{\sigma}(0)^{1}_{l_j} \right),
\e
which induces a map

\e\label{eq:Gcal_del_asc_vertic_tree}
\Gcal_{\sigma}(0)^{\kund}_{\lund} \to \Gcal_{\sigma}(0)^{\widehat{\kund}}_{\lund}  \sqcup   \coprod_{j=1}^b \Gcal_{\sigma}(0)^{1}_{l_j} ,
\e
where $\Gcal_{\sigma}(0)^{1}_{l_j}$ is either a line ($l_j=1$) or three half-lines ($l_j=2$).

Likewise, if $\kund$ is almost-vertical and $l_j=1$, one gets a similar map

\e\label{eq:Gcal_del_desc_vertic_tree}
\Gcal_{\sigma}(0)^{\kund}_{\lund} \to \Gcal_{\sigma}(0)^{\kund}_{\widehat{\lund}}  \sqcup   \coprod_{i=1}^a \Gcal_{\sigma}(0)^{k_i}_{1} .
\e

\noindent\textbf{Deletion condition:} $V$ commutes with the maps (\ref{eq:Gcal_del_asc_vertic_tree}) and (\ref{eq:Gcal_del_desc_vertic_tree}).

Finally, for the simplification relations (W) of 1-morphisms, assume $\dim \sigma = 1$, $\lund = \One_b$, then
\ea
J_{\sigma}(1)^{\kund}_{\lund} &\simeq J_{\sigma}(1)^{k_1}_{\lund} \times \cdots \times J_{\sigma}(1)^{k_a}_{\lund}\text{, and}\\
\Gcal_{\sigma}(1)^{\kund}_{\lund} &\simeq \coprod_{i=1}^{a} J_{\sigma}(1)^{k_1}_{\lund} \times \cdots \times \Gcal_{\sigma}(1)^{k_i}_{\lund} \times \cdots \times J_{\sigma}(1)^{k_a}_{\lund},
\ea
which gives a map
\e\label{eq:Gcal_W_lvertic}
\Gcal_{\sigma}(1)^{\kund}_{\lund} \to \coprod_{i=1}^{a}  \Gcal_{\sigma}(1)^{k_i}_{\lund} .
\e
Likewise, if $\kund= \One_a$, one has a similar map
\e\label{eq:Gcal_W_kvertic}
\Gcal_{\sigma}(1)^{\kund}_{\lund} \to \coprod_{j=1}^{b}  \Gcal_{\sigma}(1)^{\kund}_{l_j} .
\e

\noindent\textbf{Condition W:} $V$ commutes with  (\ref{eq:Gcal_W_lvertic}) and (\ref{eq:Gcal_W_kvertic}).

This ends the conditions defining $\Xfrak(S )$.

\section{Moduli spaces and construction of the functors}
\label{sec:Moduli_spaces_Lie}

\subsection{Moduli spaces}
\label{ssec:Moduli_spaces}

Let us now introduce the moduli space involved in Theorem~\ref{th:LieMon_to_fBialg}. Let $S\in N(\LieMon')_n$ be an $n$-simplex as in the previous section, and $V\in \Xfrak(S)$.
\begin{defi}\label{def:mod_space}
Let $\ddhat = (\sigma , \kund, \lund)$, with $\sigma = [\sigma_0, \ldots , \sigma_m]\subset [n]$; and a pair of families of critical points
\ea
x = \left\lbrace x_{ij}\right\rbrace &\in  \hspace{.5em} in (\ddhat) := \left( \left( \Crit (f_{\sigma_m}) \right)^{b}\right)^{\abs{\kund}} , \\
y = \left\lbrace y_{ij}\right\rbrace &\in out (\ddhat) := \left( \left( \Crit (f_{\sigma_0}) \right)^{\abs{\lund}}\right)^{a} .
\ea

Let $\Mcal_{\ddhat}(x,y;V)$ be the moduli space of pairs $(p,\gamma)$, where $p\in J_{\ddhat}$, and 
\e
\gamma \colon \left( \Gcal_{\ddhat} \right)_p \to G_0 \sqcup \cdots \sqcup G_n
\e
such that:
\begin{itemize}
\item $\gamma$ maps the $i$-th level to $G_{\sigma_i}$:
\e
\gamma_i \colon \left( \Gcal_{\ddhat;i} \right)_p \to G_{\sigma_i} .
\e
\item On each interval, $\gamma$ is a flowline for $V$, i.e. $ \frac{\d \gamma}{\d t} = V(\gamma)$, with $t=-h$.
\item At pairs of points $(t_i, t_{i+1})\in \left( \Gcal_{\ddhat;i} \right)_p  \times \left( \Gcal_{\ddhat;i+1} \right)_p $ coming from the same point of the grafting surface $\Hcal_{\dd}$, $\gamma$ satisfies the grafting condition
\e
\phi_{\sigma_i \sigma_{i+1}} \left( \gamma(t_{i+1}) \right) = \gamma(t_i).
\e
\item $\gamma$ satisfies a multiplicative condition at vertices of the ascending forest. That is, if $t_1, \ldots, t_d$ correspond to endpoints of the ordered incoming edges $e_1, \ldots, e_d$ of a given ascending vertex (and are on the same branch of the descending forest), and $t'$ the endpoint for the outgoing edge, then
\e
\gamma( t_1) \times  \cdots \times  \gamma(t_d ) = \gamma (t').
\e
\item $\gamma$ coincides at endpoints corresponding to a given vertex of the descending forest (and on the same branch of the ascending forest). That is, for such endpoints $t_1, \ldots, t_d, t'$:
\e
\gamma( t_1) =  \cdots =  \gamma(t_d ) = \gamma (t').
\e

\item One has the limiting conditions at $t\to \pm \infty$:
\ea
\lim_{t\to -\infty} \gamma &= x, \\
\lim_{t\to +\infty} \gamma &= y. 
\ea
Here, for $\abs{t}$ large enough, by folding all the intervals, we view $\gamma$ as a single map of the form
\ea
\gamma &\colon (-\infty, K] \to  \left( \left( G_{\sigma_m} \right)^{b}\right)^{\abs{\kund}}\text{, or} \\
\gamma &\colon [K, +\infty) \to  \left( \left( G_{\sigma_0} \right)^{\abs{\lund}}\right)^{a} .
\ea
\end{itemize}
\end{defi}

To study $\Mcal_{\ddhat}(x,y)$, one can view it as the zero set of a Fredholm section of a Banach bundle. One can also view it as an intersection of piecewise-smooth submanifolds, as a particular case of a pushforward moduli space; as for the grafted lines moduli spaces in Section~\ref{sec:informal_outline}, and as we explain below.

Given a grafted biforest $p=(\Lund, B) \in J(n)^{\kund}_{\lund}$, let 
\ea
X &= \left(( G_n)^b\right)^{\abs{\kund}}\text{, and} \\
Y &= \left(( G_0)^{\abs{\lund}}\right)^{a}.
\ea
The associated grafted graph $\Gcal_p$ provides a map 
\e
\Phi_p \colon X \to Y,
\e 
analogous to (\ref{eq:Phi_Lund_grafted_line}), i.e. reading the graph from height $h_{\max}+1$ to $h_{\min}-1$, with $h_{\max}$ and $h_{\min}$ standing respectively for the maximum and minimum of $H(\Lund) \cup h(\Vert (B) )$, and applying the following rules:
\begin{itemize}
\item when one encounters an ascending vertex, one applies the multiplication;
\item at a descending vertex, one applies the diagonal map;
\item at a grafting level, one applies the maps $\phi_{ij}$;
\item and on finite intervals of lenth $L$, one applies the time $L$ flow of $V$.
\end{itemize}  
Then, as a single map, we get as in (\ref{eq:Phi_grafted_line}),
\e\label{eq:Phi_grafted_biforest}
\Phi\colon J(n)^{\kund}_{\lund} \times \left( \left(G_n\right)^b \right)^{\abs{\kund}} \to \left( (G_0)^{\abs{\lund}}  \right)^a.
\e

\begin{lemma}
Equiping:
\begin{itemize}
\item $J_{\dd}$ with a Morse function $f_{J_{\dd}}$ with a single critical point $o$ on its interior, which is furthermore a local maximum; and a pseudo-gradient $V_{J_{\dd}}$ for it,
\item $X$ and $Y$ with the Morse functions respectively induced by $f_n$ and $f_0$; and their associated pseudo-gradients
\ea
f_{X}( \lbrace x_{ij} \rbrace ) &= \sum_{i,j} f_n (x_{ij}), \ \  V_{X}( \lbrace x_{ij} \rbrace ) = \sum_{i,j} V_n (x_{ij}); \\
f_{Y}( \lbrace y_{ij} \rbrace ) &= \sum_{i,j} f_0 (y_{ij}), \ \ V_{Y}( \lbrace y_{ij} \rbrace ) = \sum_{i,j} V_0 (y_{ij}).
\ea

\item $\widehat{X} = J_{\dd} \times X$ with $f_{\widehat{X}} = f_{J_{\dd}} + f_{X} $ and $V_{\widehat{X}} = V_{J_{\dd}} + V_{X} $
\end{itemize}
Then, with $\widehat{x} = (o,x)$, and  
$\tund = h_{\dd}^{-1} (h_{\max}+1) \subset \Gcal_p $ the level set (consisting in $b \cdot \abs{\kund}$ points), the map $ (p, \gamma) \mapsto (p, \gamma   (\tund)) $ gives a bijection from $\Mcal_{\dd}(x,y;V)$ to the pushforward moduli space of $\Phi$, $U_{\widehat{x}} \cap \Phi^{-1}(S_y) \simeq \Gamma(\Phi) \cap (U_{\widehat{x}}\times S_y )$.
\end{lemma}

\begin{proof}
 From $ \gamma   (\tund)$ one can reconstruct the upper part $\gamma_{|[h_{\max}+1, +\infty)}$ by flowing up, and $\gamma_{|(-\infty , h_{\max}+1]}$ by flowing down, and applying the various operations at vertices and grafting levels. This shows that the map in the statement has an inverse map.
\end{proof}

Observe that $J_{\dd}$ contains a collection of ``transition hyperplanes'' corresponding to where two vertices of grafting levels have the same heights. When one crosses these walls, the topological type of the corresponding grafted graph $\Gcal_p$ changes. The map $\Phi$ is smooth on the complement of these walls, but smoothness might fail at these walls. Nevertheless, the map $\Phi$ is continuous, since:
\begin{itemize}
\item the multiplication $m\colon G\times G \to G$ is associative: this ensures continuity when two vertices of the ascending forest collide.
\item the diagonal $\Delta \colon  G \to G\times G$ is coassociative: this ensures continuity when two vertices of the descending forest collide.
\item the Hopf relation holds at the monoid level, i.e. 
$\Delta \circ m = (m\times m) \circ \tau_{23} \circ (\Delta \times \Delta)$, 
where $\tau_{23} \colon G^4 \to G^4$ exchanges the second and third factors. This ensures continuity when an ascending vertex and a descending one exchange their positions (i.e. a ``Hopf pattern'' appears in the graph).
\item the maps $\phi_{ij}$ are monoid homomorphisms, this ensures continuity when a grafting level goes over an ascending vertex.
\item the maps $\phi_{ij}$ commute with $\Delta$, i.e. $\Delta \circ \phi_{ij} = (\phi_{ij} \times \phi_{ij}) \circ \Delta$. This ensures continuity when a grafting level goes over a descending vertex.
\end{itemize}
Therefore, $\Phi$ is piecewise-smooth. Moreover, $\Phi$ is smooth on each stratum, by smoothness of flows and the various maps appearing. 
It follows that $\Mcal_{\dd}(x,y;V)$ has virtual dimension
\ea
{\rm vdim} \Mcal_{\dd}(x,y;V) &= \ind(\widehat{x}) - \ind(y)
 = \ind(x) - \ind(y) + \dim J_{\dd} \\
&= \ind(x) - \ind(y) + v(\kund) + v(\lund) +n-1 .\nonumber
\ea

\begin{defi}\label{def:V_regular}
We say that $V\in \Xfrak(S)$ is \emph{regular} if for all $\ddhat, x,y$ such that ${\rm vdim}\Mcal_{\ddhat}(x,y)\leq 1$, each stratum of $\Gamma(\Phi)$ intersects $U_{\widehat{x}}\times S_y$  transversely. 
We denote $\Xfrak^{\rm reg}(S) \subset \Xfrak(S)$ the set of regular elements.
\end{defi}

Let us denote 
\e
\Gcal(\partial [n]) = \coprod_{\sigma \varsubsetneq [n]} \Gcal_{\sigma}(\dim \sigma),
\e
so that $\Gcal( [n]) =\Gcal(\partial [n]) \sqcup \Gcal_{[n]}(n )$. Let then 
\e
\Xfrak(\partial S) \subset  C^{\infty} ( \Gcal(\partial [n]) , \Xfrak (G_0) \sqcup \cdots \sqcup  \Xfrak (G_n) )
\e
be the subset satisfying all conditions of $\Xfrak( S)$. In other words, $\Xfrak(\partial S) \subset \prod_{i=0}^n \Xfrak(\partial_i S)$ is the subset such that all common boundaries agree. Let also 
\e
\Xfrak^{\rm reg}(\partial S) = \Xfrak(\partial S) \cap \left( \prod_{i=0}^n \Xfrak^{\rm reg}(\partial_i S) \right).
\e
The inclusion $\Gcal(\partial [n]) \subset \Gcal( [n])$ gives a restriction to the boundary map 
\e
R_{\partial} \colon  \Xfrak( S)\to \Xfrak (\partial S).
\e

If $W\in \Xfrak (\partial S)$, $R_{\partial}^{-1}(W) $ corresponds to maps $V\colon \Gcal_{[n]}(n) \to \coprod_i \Xfrak(G_i)$ with behaviour over boundaries and ends of $J_{[n]}(n)$ (i.e. boundaries of $\left(\overline{J}_{[n]}(n) \right) _{\leq 1}$) prescribed by $W$.

\begin{lemma}\label{lem:reg_extention}
Assume that $W\in \Xfrak^{\rm reg} (\partial S)$. Then $\Xfrak^{\rm reg} ( S) \cap R_{\partial}^{-1}(W) $ is a comeagre subset of $R_{\partial}^{-1}(W) $.
\end{lemma}
\begin{proof}

The standard transversality argument based on Sard-Smale's theorem applies in our setting. Let us briefly sketch it, we refer for example to \cite{FloerHoferSalamon,AudinDamian} for more details. One first forms a universal moduli space
\e
F^{-1}(0) = \bigcup_{V\in R_{\partial}^{-1}(W)} \Mcal_{\ddhat}(x,y;V) \times \lbrace V \rbrace \subset \Bcal_{\ddhat}(x,y) \times R_{\partial}^{-1}(W),
\e
where $\Bcal_{\ddhat}(x,y)$ stands for graphs as in Definition~\ref{def:mod_space} but without the flow equation, and $F(\gamma,V)= \frac{\d \gamma}{\d t} - V$ is the operator bringing back the flow equation.

If one shows that $DF_{|F^{-1}(0)}$ is surjective, then $F^{-1}(0)$ is smooth, and the result follows from Sard-Smale's theorem applied to the projection $F^{-1}(0) \to R_{\partial}^{-1}(W)$.

Surjectivity of $DF_{|F^{-1}(0)}$ involves a duality argument: an element in the cokernel on the one hand satisfies a unique continuation principle, and on the other hand must be identically zero on open sets where perturbations of $V$ are unrestricted. The next three observations permit to conclude the proof.

First, \emph{in most components} of $\Gcal_{[n]}(n)$, over the interior of $\left(\overline{J}_{[n]}(n) \right) _{\leq 1}$, the unrestricted perturbation region is $\Gcal^{\rm pert}$, and intersects every connected components of fibers.

Second, over the boundary of $\left(\overline{J}_{[n]}(n) \right) _{\leq 1}$, i.e. the boundary and the ends of ${J}_{[n]}(n)$, since $W$ is regular, the boundaries and ends of the moduli spaces $\partial \overline{\Mcal}_{\ddhat}(x,y;V)$ are transversally cut out. Transversality for these is a stronger condition than transversality as elements of $\Mcal_{\ddhat}(x,y;V)$, which can be seen as a parametrized version of transversality, and is an open condition. It follows that $\Mcal_{\ddhat}(x,y;V)$ is automatically transverse in a neighborhood of those points.

Third, on some components of $\Gcal_{[n]}(n)$, perturbations are completely restricted: they are determined by $V$ on other components. In those cases, the corresponding moduli spaces can be expressed as products of moduli spaces corresponding to other components, as we shall see in Propositions~\ref{prop:descr_bdary_mod_spaces},~
\ref{prop:mod_space_deletion} and 
\ref{prop:mod_space_W}. Therefore they are transversely cut out as soon as the other component ones are.
\end{proof}

\begin{defi}\label{def:Liepert}
Let $\LieMonpert$ be the simplicial set whose $n$-simplices consist in $\widehat{S} = (S, V)$, where $S$ is an $n$-simplex of $N(\LieMon ')$, and $V\in \Xreg(S)$.

A \emph{coherent choice of perturbations} is a simplicial section $\mathbb{V} \colon \LieMon \to \LieMonpert$. Concretely, it is a choice for any $S$ of $V\in \Xreg (S)$, consistent with restrictions to boundaries of $S$.
\end{defi}

\begin{lemma}\label{lem:exist_coh_choices_pert}
Coherent choice of perturbations exist.
\end{lemma}
\begin{proof}
By induction one can construct $\mathbb{V}_n \colon (\LieMon)_n \to (\LieMonpert)_n$.

Assume $\mathbb{V}_{n-1} $ is constructed, we want to find

\e
V \colon \Gcal([n]) = \Gcal(\partial [n]) \sqcup \Gcal_{[n]}(n) \to \coprod_i \Xfrak(G_i)
\e
Since $\mathbb{V}_n$ is simplicial, its restriction to $ \Gcal(\partial [n])$ is determined by  $\mathbb{V}_{n-1} $. From Lemma~\ref{lem:reg_extention}, it extends to a regular $V$ on $\Gcal_{[n]}(n)$.

\end{proof}

\begin{prop}\label{prop:Liepert_wKan}
$\LieMonpert$ is a weak Kan complex.
\end{prop}
\begin{proof} This follows from Lemma~\ref{lem:reg_extention}, applied twice. Denote

\e
\Gcal(\partial [n]) = \coprod_{\sigma \varsubsetneq [n]} \Gcal_{\sigma}(\dim \sigma), \ \ 
\Gcal(\Lambda^n_i) = \coprod_{\begin{subarray}{c} \sigma \varsubsetneq [n] \\  \sigma \neq \partial_i [n] \end{subarray}} \Gcal_{\sigma}(\dim \sigma) .
\e
Consider an inner horn of $\LieMonpert$, i.e. an $n$-simplex $S$ of $N(\LieMon)$, Morse functions and pseudo-gradients $\left\lbrace f_i, V_i \right\rbrace_i $, and a regular 
\e
V\colon \Gcal (\Lambda^n_i) \to \coprod_i \Xfrak (G_i).
\e
The lemma applied to $V$ on $\Gcal (\partial \Lambda^n_i) =\Gcal (\partial (\partial_i [n])) $ permits to extend $V$ to 
\e
\Gcal (\partial [n]) =  \Gcal(\Lambda^n_i) \sqcup  \Gcal_{\partial_i [n]} (n-1).
\e
Applying the lemma a second time permits to extend from $\partial [n]$ to $ [n]$.

\end{proof}

\begin{prop}\label{prop:descr_bdary_mod_spaces}
Fix an $n$-simplex $S \in N(\LieMon')_n$, and $V\in \Xreg (S)$ (which we live implicit). Let $\ddhat =(\sigma, \kund, \lund)$ and $x,y$ critical points as in Definition~\ref{def:mod_space}. Then,
\begin{itemize}
\item if  ${\rm vdim} \Mcal_{\ddhat}(x,y) <0$, $\Mcal_{\ddhat}(x,y) = \emptyset$,
\item if  ${\rm vdim} \Mcal_{\ddhat}(x,y)=0$, $\Mcal_{\ddhat}(x,y)$ is a compact 0-manifold, oriented relatively to $x$ and $y$.
\item if  ${\rm vdim} \Mcal_{\ddhat}(x,y)=1$, $\Mcal_{\ddhat}(x,y)$ compactifies to a piecewise-smooth 1-manifold $\overline{\Mcal}_{\ddhat}(x,y)$, oriented relatively to $x,y$, and whose boundary is given by:
\ea\label{eq:bdary_mod_spaces}
\partial (\overline{\Mcal}_{\ddhat}(x,y)) &= \coprod_{x'} (-1)^{} \cdot \Mcal_{\partial}(x,x') \times \Mcal_{\ddhat}(x',y) \\
&\sqcup \coprod_{y'}   (-1)^{\dim J_{\dd}} \cdot \Mcal_{\ddhat}(x,y')\times \Mcal_{\partial}(y',y) \nonumber\\
&\sqcup \coprod_{i=1}^{m-1}(-1)^{i} \cdot  \Mcal_{\partial_i \ddhat}(x,y) \nonumber\\
&\sqcup \coprod_{\begin{subarray}{c}   \ddhat = \ddhat^1 \sharp \ddhat^0 \\ z \end{subarray}} (-1)^{\rho^{0,1}} \cdot  \Mcal_{\ddhat^0}(x,z)\times \Mcal_{\ddhat^1}(z,y) , \nonumber
\ea
Where we recall that $\partial_i \ddhat = (\partial_i \sigma, \kund, \lund)$; 
 $x',y',z$ are critical points in the appropriate products of monoids, and of the appropriate Morse indices (so that the moduli spaces appearing all have virtual dimension zero).
\end{itemize}
Furthermore,

\begin{itemize}
\item If $\dim \sigma = 0$, $\lund = \One_b$, and $\kund = (1, \ldots 1, k_i, 1, \ldots 1)$, decomposing $x$ and $y$ into their components corresponding to the $i$-th part of $\kund$ and the vertical part of $\kund$:
\ea
x' \in \left( \left( \Crit f_{\sigma_0}  \right)^b \right)^{a-1} ,  & \ \ \  y' \in \left( \left( \Crit f_{\sigma_0}  \right)^b \right)^{a-1} ,  \\
x'' \in \left( \left( \Crit f_{\sigma_0}  \right)^b \right)^{k_i} ,  & \ \ \  y'' \in \left(  \Crit f_{\sigma_0}  \right)^b  ;
\ea
One has:
\e\label{eq:mod_spaces_atilde_1}
\Mcal_{\ddhat}(x,y) = \begin{cases} (\Mcal_{\sigma_0})^{k_i}_{\One_b}(x'', y'') &\text{if } x'=y', \\
\emptyset &\text{otherwise.} \end{cases}
\e
Likewise exchanging the roles of $\kund$ and $\lund$.

\item If $\dim \sigma = 0$, $\lund = \One_b$, and $\tilde{a} \geq 2$, then 
\e\label{eq:Mcal_deg_emptyset}
\Mcal_{\ddhat}(x,y) =  \emptyset .
\e
Likewise exchanging the roles of $\kund$ and $\lund$.
\end{itemize}
\end{prop}

\begin{proof} The statements for ${\rm vdim} \Mcal_{\ddhat}(x,y) \leq 0$, as well as piecewise smoothness and relative orientability when ${\rm vdim} \Mcal_{\ddhat}(x,y) =1$; follow from the fact that $V$ is regular.

(\ref{eq:Mcal_deg_emptyset}) is a tautology since in that case we defined $\Gcal = \emptyset$. What is not a tautology is that there is no corresponding contributions in the boundary equation (\ref{eq:bdary_mod_spaces}), as we explain below.

In (\ref{eq:bdary_mod_spaces}), the inclusion $ \partial (\overline{\Mcal}_{\ddhat}(x,y)) \subset RHS$ follows from the compactness result on flowlines, and codimension considerations: breaking in negative codimension doesn't occur since $V$ is regular. In particular, if either $\ddhat^{0}$ or $\ddhat^{1}$ is of the form of (\ref{eq:Mcal_deg_emptyset}), the corresponding moduli space has extra symmetries, and projects to a moduli space of negative dimension, which must be empty. 

More specifically, assume for example that $\ddhat^{0} = ([\sigma_0],\kund, \One_b)$, with $\tilde{a}\geq 1$, 
 and that a sequence $\gamma_n$ tends to a broken grafted flowgraph $(\gamma^0, \gamma^1)$, with $\gamma^0$ of type $\ddhat^{0}$.

On the one hand, for $\gamma^0$ and $\gamma^1$ to exist, they must both have virtual dimension zero. 
On the other hand, decompose $\gamma^0$ in its vertical part ($k_i=1$), and parts corresponding to each nonvertical tree of $\kund$:
\e
\gamma^0 = \gamma_{\rm vert} \sqcup \gamma_1 \sqcup \cdots \sqcup\gamma_{\tilde{a}}.
\e
This gives a point in the product moduli space
\e
 (\gamma_1 , \ldots , \gamma_{\tilde{a}}) \in \Mcal_{\ddhat_1}(x_1 , y_1) \times \cdots \times \Mcal_{\ddhat_{\tilde{a}}}(x_{\tilde{a}} , y_{\tilde{a}}),
\e
and this product has virtual dimension $1-{\tilde{a}}$, since $J_{\dd_1} \times \cdots \times J_{\dd_{\tilde{a}}}  \simeq \left(J^{\kund^0}_{\lund^0} /\rr\right)/\rr^{1-{\tilde{a}}} $. Therefore, if ${\tilde{a}}\geq 2$, it is empty, and such breaking does not occur. If ${\tilde{a}}=1$, the vertical part $\gamma_{\rm vert}$ must be constant, which gives (\ref{eq:mod_spaces_atilde_1}).

Finally, the inclusion $ \partial (\overline{\Mcal}_{\ddhat}(x,y)) \supset RHS$ in (\ref{eq:bdary_mod_spaces}) is the standard gluing argument: broken graphs can be glued, and the pre-gluing map on moduli spaces is defined using the abstract gluing maps $g_{\dd^0, \dd^1}$ constructed in Section~\ref{sec:higher_bimultipl}. This also shows that the orientation of $g_{\dd^0, \dd^1}$ corresponds to  the orientations as boundary components.
\end{proof}

\begin{prop}\label{prop:mod_space_deletion}
Assume that $\ddhat = ([\sigma_0],\kund,\lund)$ is as in the setting of (\ref{eq:D_kund_i_lund}): $\lund$ is almost vertical,  $k_i = 1$, and  $\widehat{\kund} = (k_1, \ldots , \widehat{k_i} , \ldots, k_a)$. 
Decompose $x$ to $\widehat{x}$ corresponding to $\widehat{\kund}$, and to $x_1, \ldots , x_b$ for the vertical tree at position $i$. 
Similarly, decompose $y$ to $\widehat{y}$ and  $y_1, \ldots , y_b$. Then,
\e
\Mcal_{\ddhat}(x,y) \simeq (\Mcal_{\sigma})^{\widehat{\kund}}_{\lund} (\widehat{x},\widehat{y}) \times (\Mcal_{\sigma})^{1}_{l_1} (x_1,y_1) \times \cdots \times  (\Mcal_{\sigma})^{1}_{l_b} (x_b,y_b) .
\e

Likewise, if $\ddhat = ([\sigma_0],\kund,\lund)$ is as in  (\ref{eq:D_kund_lund_i}), then with analogous decompositions of $x,y$:
\e
\Mcal_{\ddhat}(x,y) \simeq (\Mcal_{\sigma})^{\kund}_{\widehat{\lund}} (\widehat{x},\widehat{y}) \times (\Mcal_{\sigma})^{k_1}_{1} (x_1,y_1) \times \cdots \times  (\Mcal_{\sigma})^{k_a}_{1} (x_a,y_a) .
\e
\end{prop}
\begin{proof}
The key observation here is that since $\lund$ is almost vertical, one has 
\e
J(0)^{\kund}_{\lund} \simeq J(0)^{\widehat{\kund}}_{\lund} \simeq J(0)^{\widehat{\kund}}_{\lund} \times  J(0)^{1}_{l_1} \times \cdots \times  J(0)^{1}_{l_b}.
\e

Let $\gamma \in \Mcal_{\ddhat}(x,y)$, and decompose it as $\gamma = \widehat{\gamma} \sqcup \gamma_1 \sqcup \cdots \sqcup \gamma_b$. From the fact that the total dimension is zero, and from the above observation, it follows that $ \widehat{\gamma}$, $ \gamma_1$, ..., $\gamma_b$ all have virtual dimension zero, therefore they define a point in $(\Mcal_{\sigma})^{\widehat{\kund}}_{\lund} (\widehat{x},\widehat{y}) \times (\Mcal_{\sigma})^{1}_{l_1} (x_1,y_1) \times \cdots \times  (\Mcal_{\sigma})^{1}_{l_b} (x_b,y_b)$, since the pseudo-gradient $V$ satisfies the appropriate condition. Conversely, by forming the union, a point in $(\Mcal_{\sigma})^{\widehat{\kund}}_{\lund} (\widehat{x},\widehat{y}) \times (\Mcal_{\sigma})^{1}_{l_1} (x_1,y_1) \times \cdots \times  (\Mcal_{\sigma})^{1}_{l_b} (x_b,y_b)$ defines a point in $\Mcal_{\ddhat}(x,y)$.

\end{proof}

\begin{prop}\label{prop:mod_space_W}(Identities for simplification relations W) Assume $\sigma = [\sigma_0, \sigma_1]$ and $\lund = \One_b$. 
Decompose $x \in  in (\ddhat)$ to $x_1, \ldots, x_a$ and $y \in  out (\ddhat)$ to $y_1, \ldots, y_a$, with:
\ea
x_i \in \left( \left( \Crit f_{\sigma_0}  \right)^b \right)^{k_i} ,  & \ \ \  y_i \in \left( \Crit f_{\sigma_1}  \right)^b    .
\ea
Then one has:
\e
(\Mcal_{\sigma})^{\kund}_{\One_b} (x,y) \simeq (\Mcal_{\sigma})^{k_1}_{\One_b} (x_1,y_1) \times \cdots \times  (\Mcal_{\sigma})^{k_a}_{\One_b} (x_a,y_a) .
\e
\end{prop}
\begin{proof}Here, the key observation is that
\e
J(1)^{\kund}_{\One_b} \simeq J(1)^{k_1}_{\One_b} \times \cdots \times  J(1)^{k_a}_{\One_b}.
\e
The proof is then similar to Proposition~\ref{prop:mod_space_deletion}: the bijection is given by decomposing the 1-grafted graph according to $\kund$.
\end{proof}

\subsection{The functor from $\LieMon$}
\label{ssec:fun_from_LieMon}

Let us define the simplicial map $\LieMonpert \to \fBialg$. Let $(S,V)$ be an $n$-simplex in $\LieMonpert$. 
For $\ddhat$ and $x,y$ as in Definition~\ref{def:mod_space}, by Proposition~\ref{prop:descr_bdary_mod_spaces}, one can consider the normalization 
\e
\abs{\Mcal_{\ddhat}(x,y)}_R \colon \abs{x}_R \to \abs{y}_R .
\e
\begin{theo}\label{th:fun_LieMon_final}
The assignment 
\ea
(G,f,V) &\mapsto CM_*(G) ,\\
(S,V) &\mapsto \lbrace \varphi_{\ddhat} \rbrace;
\ea
with:
\e
\varphi_{\ddhat} = \sum_{\ind(y) = \ind(x) +\dim J_{\dd}} \abs{\Mcal_{\ddhat}(x,y)}_R \colon CM_*(G_{\sigma_m}) \to CM_*(G_{\sigma_m}) .
\e
defines a simplicial map $\LieMonpert \to \fBialg$.
\end{theo}
\begin{proof}
From (\ref{eq:bdary_mod_spaces}), it follows that $\varphi_{\ddhat}$ satisfies the coherence relation (\ref{eq:cohrel_R_ddhat}), i.e. this defines a simplicial map to $\PreBialg$. From (\ref{eq:mod_spaces_atilde_1}), (\ref{eq:Mcal_deg_emptyset}) and Propositions~\ref{prop:mod_space_deletion} and \ref{prop:mod_space_W}, it follows that this maps lands in $\fBialg$.
\end{proof}

\subsection{Proof of the Corollaries}
\label{ssec:proof_cor}

\begin{proof}[Proof of Corollary~\ref{cor:Manact_to_uBimodact}] Recall that one has a functor $\Manact \to \LieMon$ given by (\ref{eq:Manact_to_LieMon}). One can therefore compose this functor with the one of Theorem~\ref{th:LieMon_to_fBialg}. For an $n$-simplex in the nerve $N(\Manact)$
\e
\left\lbrace  T_i = (G_i, X_i, H_i)\ ,\ \Phi_{i(i+1)} = (\varphi_{i(i+1)}, f_{i(i+1)}, \psi_{i(i+1)}) \right\rbrace,
\e
we obtain a collection of operations, with $\sigma = [\sigma_0, \ldots, \sigma_d ]$:
\ea
(\alpha_\sigma)^{\kund}_{\lund} &\colon (A_{\sigma_d}^{b})^{\abs{\kund}} \to (A_{\sigma_0}^{\abs{\lund}})^{a}\text{, with} \\
A_{i} &= CM_*(G_i) \oplus CM_*(X_i) \oplus CM_*(H_i) \oplus CM_*(pt). \label{eq:Ai_decomp}
\ea

Observe now that the operations $(\mu_\sigma)^{\kundb}_{\lund}$ one wants to get correspond to blocks of $(\alpha_\sigma)^{\kund}_{\lund}$  w.r.t. the decomposition (\ref{eq:Ai_decomp}).
\end{proof}

\begin{proof}[Proof of Corollary~\ref{cor:fun_Man_Kom}] Now embed $\Man \hookrightarrow \Manact$ via $X \mapsto (1,X,1)$, and apply the functor of Corollary~\ref{cor:Manact_to_uBimodact}. For any $n$-simplex in $N(\Man)$, one gets operations of the form $(\mu_\sigma)^{\kundb}_{\lund}$. These operations for $\kundb = (0|1|0)$ correspond to the $\infty$-functor of Corollary~\ref{cor:fun_Man_Kom}.
\end{proof}

\section{Relation with gauge theory and Floer theory}
\label{sec:rel_other_appr}

We now put our results in a broader perspective. They fit in the following diagram, which we will explain in the next subsections.

\begin{center}
\begin{tikzcd}
& \color{blue} \Man \ar[r, blue, "CM_*"] \ar[d, blue]\ar[drr, blue, dashrightarrow]
 & \color{blue}\fCoalg \ar[d, blue, crossing over] \ar[drr, blue, dashrightarrow]
  & \Lie \Gcal r \ar[r, "CM_*"] \ar[d, crossing over]
   & \fBialg \ar[d]
    & (\infty,1) \\
\color{blue}\Cob_{2..4} \ar[r, blue, "\Don_{2..4}"] \ar[drr, blue, dashrightarrow] 
& \color{blue}\Symp  \ar[r, blue, "\Fuk"] \ar[drr, blue, dashrightarrow]
 & \color{blue}{A_\infty}\text{-}\mathcal{C}at \ar[drr, blue, dashrightarrow]\ar[from=ll, blue, crossing over, bend right=30, ""']  
  & \Lier \ar[r, "CM_*"] \ar[d, crossing over]
   & \uBimod \ar[d]
    & (\infty,2) \\
&
 &\Cob_{1..4} \ar[r, "\Don_{1..4}"] \ar[rr, bend right=30, ""']
  & \Ham \ar[r, "\Fuk"]
   & \dCat
    & (\infty,3) 
\end{tikzcd}
\end{center}

\subsection{Categorifications and extended TFT}
\label{ssec:higher_ext_extended_TFT}

The functors of Theorem~\ref{th:LieMon_to_fBialg} (restricted to Lie groups) and Corollary\ref{cor:fun_Man} correspond to the first line of the above diagram. Both should be part of a common (partial) $(\infty,2)$-functor ($CM_*$ on the second line) that should categorify the functor of Corollary~\ref{cor:Manact_to_uBimodact}, as explained below. And this categorification should be seen as a Morse counterpart to a Fukaya category $(\infty,3)$-functor ($\Fuk$ on the third line).

In \cite{ham} we introduced two ``partial'' 2-categories $\Lier$ and $\Ham$, which are 2-categorical analogs of the Moore-Tachikawa category \cite{MooreTachikawa}, respectively in the non-symplectic and in the real symplectic setting. 
Oversimplifying, these include elementary diagrams as below:

\begin{center}
\begin{tikzcd} 
& \Lier \arrow[rr, "T^*"]  &  & \Ham & \\
G \arrow[r, bend left=50, "X", ""{name=U, below}]
\arrow[r, bend right=50, "Y"{below},  ""{name=D}]
& H 
\arrow[Rightarrow, from=U, to=D, "C"] 
& \mapsto
&^G \arrow[r, bend left=50, "M", ""{name=U, below}]
\arrow[r, bend right=50, "N"{below},  ""{name=D}]
& H ,
\arrow[Rightarrow, from=U, to=D, "L"] 
\end{tikzcd}
\end{center}
\begin{itemize}
\item In both $\Lier$ and $\Ham$, objects are compact Lie groups.
\item 1-morphisms from $G$ to $H$ in $\Lier$ are compact smooth manifold with a $G\times H$-action; and in $\Ham$ 
are symplectic manifolds equipped with a Hamiltonian action of $G\times H$.
\item 2-morphisms in $\Lier$ are $(G\times H)$-invariant correspondences (i.e. subsets $C\subset X\times Y$); and 2-morphisms in $\Ham$ are $(G\times H)$-equivariant Lagrangian correspondences $L\subset M^- \times N$.

\item 1-morphisms compose via quotienting (resp. symplectic reduction) by the diagonal action:
\begin{center}
\begin{tikzcd}
G_0  \ar[r, "X_{01}"]\ar[rr, bend right=25,"(X_{01}\times X_{12})/{G_1}"'] 
& G_1 \ar[r, "X_{12}"]
 & G_2 
  & 
   & G_0  \ar[r, "M_{01}"]\ar[rr, bend right=25,"(M_{01}\times M_{12})\red G_1"'] 
    & G_1 \ar[r, "M_{12}"]
     & G_2 .
\end{tikzcd}
\end{center}

\item 2-morphisms compose as correspondences.
\end{itemize}

The category of compact Lie groups $\Lie \Gcal r$ embeds in $\Lier$; and $\Lier$ embeds in $\Ham$ via a ``cotangent 2-functor'' $T^* \colon \Lier \to \Ham$ that sends objects, 1-morphisms and 2-morphisms respectively to themselves, their cotangent bundle, and their conormal bundle.

Moreover, $\Man$ embeds in the endomorphisms of the trivial group $End_{\Lier}(1)$, which we indicate as a dotted arrow \begin{tikzcd}[cramped, sep=small] 
\Man \ar[r, dashrightarrow]
  & \Lier
\end{tikzcd}. Therefore, $\Lier$ contains both $\Man$ and $\Lie$.

Furthermore, $\Lier$ categorifies $\Manact$, in the sense that objects and 1-morphisms of $\Manact$ can be seen respectively as 1-morphisms and 2-morphisms in $\Lier$. We expect that $\uBimodact$ can be categorified in a similar way to an $(\infty,2)$-category $\uBimod$, which should contain both $\fCoalg$ and $\fBialg$, and should be the target of a functor from $\Lier$ that categorifies the functor of Corollary~\ref{cor:Manact_to_uBimodact}.

Moving down to the third line, $f$-bialgebras, $u$-bimodules and their morphisms have categorical counterparts that we introduced in \cite[Sec.~3.6,~3.7]{biass}; in the same way than the \Ainf -(co)category $\Fuk(M)$ is a categorical counterpart of the \Ainf-coalgebra $CM_*(X)$. We believe these counterpart should form (some sort of) $(\infty,3)$-category $\dCat$; and that an $(\infty,3)$-functor, whose effect on simple diagrams is described by \cite[Conj.~B]{biass} (or a wrapped variant from \cite[Conj.~C]{biass}), should exist.

In \cite{ham} we constructed a (partial, quasi-) 2-functor $\Don_{1..3}\colon \Cob_{1..3} \to \Ham$\footnote{We are oversimplifying here, it takes values in a ``completion'' $\Hamhat$, see \cite{ham} for more details.}. We believe it should extend to dimension four, and composing it with the above $(\infty,3)$-functor should produce an extended TFT that should contain (equivariant versions of) Instanton homology the Donaldson polynomials.

\subsection{Relation with work of Wehrheim, Woodward, Ma'u and Bottman}
\label{ssec:Rel_MWWB}

What we described above is an extension down to dimension 1 of \WW's ``Floer Field theory'' \cite{WWfft,Wehrheimphilo}. They suggest that Weinstein's symplectic category $\Symp$ could be seen as a 2-category, with Lagrangian Floer  homology groups as 2-morphism spaces. Following the Atiyah-Floer conjecture, they suggest there should be an extended TFT in dimensions 2..4 with values in this 2-category, corresponding to Donaldson's invariants and Instanton homology in dimensions 4 and 3 respectively.

In particular, they show that Lagrangian correspondences induce functors between Donaldson categories. In \cite{MauWehrheimWoodward}, as a chain-level counterpart, they show it induces an \Ainf -functor between (some variation of) Fukaya categories. They suggest that $\Symp$ should form some sort of 2-category, with Lagrangian Floer chain groups as 2-morphisms. Such a structure, referred to as $(A_{\infty},2)$-category, is currently under construction \cite{BottmanCarmeli,Bottman_flowcat}.

The starting point of these constructions is the use of quilted discs as in the right side of Figure~\ref{fig:quilts_multiplihedra}, which realize the multiplihedron \cite{MauWoodwardrealiz}, and can be used to define \Ainf -functors. Bottman considers generalizations of those, called witch balls, in order to define the 2-associahedron \cite{Bottman_2_ass}, which dictates the algebraic structure of $(A_{\infty},2)$-categories.

Our approach is a priori different: the natural counterpart of our grafted trees are quilts as in the left of Figure~\ref{fig:quilts_multiplihedra}, and are different than those of \cite{MauWehrheimWoodward}; even though they both realize the multiplihedron. However, the corresponding generalization, i.e. $J(n)_{\lund}$, is combinatorially different from Bottman's 2-associahedron. It would be interesting to compare the resulting algebraic structures: are they equivalent, in a certain sense?

\begin{figure}[!h]
    \centering
    \def\svgwidth{.80\textwidth}
    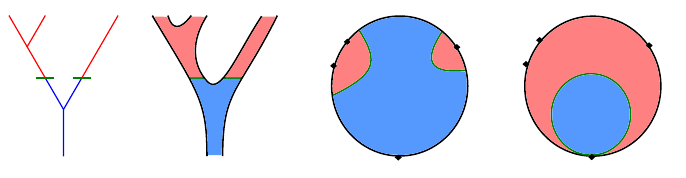
      \caption{From left to right: a grafted tree, its quilted counterpart, and the same quilt in the disc model; which is different from the quilted discs in \cite{MauWoodwardrealiz}.}
      \label{fig:quilts_multiplihedra}
\end{figure}

Let us also mention Fukaya's work \cite{Fukaya_functor}, that should informally be seen as a 2-functor $\Symp \to {A_\infty}\text{-}\mathcal{C}at$. Fukaya's construct is more indirect in nature, and relies on a systematic use of the Yoneda embedding. It is unclear to us how it relates to our construction.


%

\bibliographystyle{alpha}
\bibliography{biblio}

\end{document}

%% file: rectangular_box.pdf_tex
\begingroup%
  \makeatletter%
  \providecommand\color[2][]{%
    \errmessage{(Inkscape) Color is used for the text in Inkscape, but the package 'color.sty' is not loaded}%
    \renewcommand\color[2][]{}%
  }%
  \providecommand\transparent[1]{%
    \errmessage{(Inkscape) Transparency is used (non-zero) for the text in Inkscape, but the package 'transparent.sty' is not loaded}%
    \renewcommand\transparent[1]{}%
  }%
  \providecommand\rotatebox[2]{#2}%
  \newcommand*\fsize{\dimexpr\f@size pt\relax}%
  \newcommand*\lineheight[1]{\fontsize{\fsize}{#1\fsize}\selectfont}%
  \ifx\svgwidth\undefined%
    \setlength{\unitlength}{226.77165354bp}%
    \ifx\svgscale\undefined%
      \relax%
    \else%
      \setlength{\unitlength}{\unitlength * \real{\svgscale}}%
    \fi%
  \else%
    \setlength{\unitlength}{\svgwidth}%
  \fi%
  \global\let\svgwidth\undefined%
  \global\let\svgscale\undefined%
  \makeatother%
  \begin{picture}(1,0.5625)%
    \lineheight{1}%
    \setlength\tabcolsep{0pt}%
    \put(0,0){\includegraphics[width=\unitlength,page=1]{rectangular_box.pdf}}%
    \put(0.17909219,0.44918361){\color[rgb]{0,0,0}\makebox(0,0)[lt]{\lineheight{1.25}\smash{\begin{tabular}[t]{l}$h$\end{tabular}}}}%
    \put(0.11575521,0.11601567){\color[rgb]{0,0,0}\makebox(0,0)[lt]{\lineheight{1.25}\smash{\begin{tabular}[t]{l}$h_{01}$\end{tabular}}}}%
    \put(0.11435078,0.17545657){\color[rgb]{0,0,0}\makebox(0,0)[lt]{\lineheight{1.25}\smash{\begin{tabular}[t]{l}$h_{12}$\end{tabular}}}}%
    \put(0.11498472,0.28707252){\color[rgb]{0,0,0}\makebox(0,0)[lt]{\lineheight{1.25}\smash{\begin{tabular}[t]{l}$h_{23}$\end{tabular}}}}%
    \put(0,0){\includegraphics[width=\unitlength,page=2]{rectangular_box.pdf}}%
    \put(0.0504728,0.3670994){\color[rgb]{0,0,0}\makebox(0,0)[lt]{\lineheight{1.25}\smash{\begin{tabular}[t]{l}$x$\end{tabular}}}}%
    \put(0.04807453,0.01379086){\color[rgb]{0,0,0}\makebox(0,0)[lt]{\lineheight{1.25}\smash{\begin{tabular}[t]{l}$y$\end{tabular}}}}%
  \end{picture}%
\endgroup%

%% file: quilts_multiplihedra.pdf_tex
\begingroup%
  \makeatletter%
  \providecommand\color[2][]{%
    \errmessage{(Inkscape) Color is used for the text in Inkscape, but the package 'color.sty' is not loaded}%
    \renewcommand\color[2][]{}%
  }%
  \providecommand\transparent[1]{%
    \errmessage{(Inkscape) Transparency is used (non-zero) for the text in Inkscape, but the package 'transparent.sty' is not loaded}%
    \renewcommand\transparent[1]{}%
  }%
  \providecommand\rotatebox[2]{#2}%
  \newcommand*\fsize{\dimexpr\f@size pt\relax}%
  \newcommand*\lineheight[1]{\fontsize{\fsize}{#1\fsize}\selectfont}%
  \ifx\svgwidth\undefined%
    \setlength{\unitlength}{328.81889764bp}%
    \ifx\svgscale\undefined%
      \relax%
    \else%
      \setlength{\unitlength}{\unitlength * \real{\svgscale}}%
    \fi%
  \else%
    \setlength{\unitlength}{\svgwidth}%
  \fi%
  \global\let\svgwidth\undefined%
  \global\let\svgscale\undefined%
  \makeatother%
  \begin{picture}(1,0.25862069)%
    \lineheight{1}%
    \setlength\tabcolsep{0pt}%
    \put(0,0){\includegraphics[width=\unitlength,page=1]{quilts_multiplihedra.pdf}}%
    \put(0.17168799,0.11246753){\color[rgb]{0,0,0}\makebox(0,0)[lt]{\lineheight{1.25}\smash{\begin{tabular}[t]{l}$\to$\end{tabular}}}}%
    \put(0.39259647,0.11232799){\color[rgb]{0,0,0}\makebox(0,0)[lt]{\lineheight{1.25}\smash{\begin{tabular}[t]{l}$\to$\end{tabular}}}}%
    \put(0.70214755,0.11267479){\color[rgb]{0,0,0}\makebox(0,0)[lt]{\lineheight{1.25}\smash{\begin{tabular}[t]{l}$\neq$\end{tabular}}}}%
  \end{picture}%
\endgroup%